\documentclass[12pt]{amsart}
\usepackage{geometry}
\usepackage[T1]{fontenc}      
\usepackage[english]{babel} 
\usepackage[utf8]{inputenc}
\usepackage{t1enc,lmodern, url, mathtools,amsmath}
\usepackage{amsfonts,amssymb,amsthm,dsfont,amsmath,hyperref}
\usepackage{palatino}
\usepackage{enumitem}
\usepackage{tikz, tikz-cd}
\usepackage{graphicx}
\usepackage{epstopdf}

\numberwithin{equation}{section}

\newcommand{\p}{\mathbb{P}}

\newcommand{\trN}{\operatorname{tr}_N}
\newcommand{\ind}{\mathbf{1}}

\newcommand{\R}{\mathbb{R}}
\newcommand{\E}{\mathbb{E}}
\newcommand{\C}{\mathbb{C}}

\newtheorem{theorem}[equation]{Theorem}

\newtheorem{proposition}[equation]{Proposition}
\newtheorem{definition}[equation]{Definition}
\newtheorem{assumptions}[equation]{Assumptions}
\newtheorem{lemma}[equation]{Lemma}
\newtheorem{corollary}[equation]{Corollary}
\newtheorem{remark}[equation]{Remark}
\newtheorem{example}[equation]{Example}
\newtheorem*{assumption}{Assumption}
\newtheorem{question}{Question}

\author{Lucas Babet}
\address{École Normale Supérieure}
\email{lucas.babet@ens.psl.eu}

\author{Ionel Popescu}
\address{University of Bucharest, Faculty of Mathematics and Computer Science;
Institute of Mathematics ``Simion Stoilow'' of the Romanian Academy}
\email{ionel.popescu@fmi.unibuc.ro}
\email{ionel.popescu@imar.ro}
\date{December 2025}
\title{Tridiagonal random matrices, an analytic approach}

\begin{document}
\maketitle

\begin{abstract}
In this paper, we study the limiting distribution of the eigenvalues for random tridiagonal matrix models. In the paper \cite{P09}, the limiting distribution is well described by its moments. Here, an analytical approach allows us, as in the case of Wigner matrices, to relax the assumptions on the random variables. With this method, we proved the convergence of the spectral distribution under an assumption on the second moment. We discuss also about an algebraic approach for the tridiagonal models, which are more complicated than the classic freeness.
\end{abstract}

\tableofcontents

\section{Introduction}

Random matrix theory encompasses both dense ensembles and reduced matrix
models.  The subject originates in Wigner's work
\cite{Wigner1958}; see also Mehta's monograph
\cite{mehta2004random}.  Trotter \cite{Trotter1984EigenvalueDO} proved the
semicircle law through tridiagonalization.  Dumitriu and Edelman later
introduced the tridiagonal $\beta$-Hermite model \cite{DE02},
\[
\frac{1}{\sqrt{2}} \begin{bmatrix}
\mathcal{N}(0,2) & \chi_{(n-1)\beta} & 0 & \cdots & 0 & 0 \\
\chi_{(n-1)\beta} & \mathcal{N}(0,2) & \chi_{(n-2)\beta} & \ddots & \vdots & \vdots \\
0 & \chi_{(n-2)\beta} & \mathcal{N}(0,2) & \ddots & 0 & 0 \\
\vdots & \ddots & \ddots & \ddots & \chi_{2\beta} & 0 \\
0 & \cdots & 0 & \chi_{2\beta} & \mathcal{N}(0,2) & \chi_{\beta} \\
0 & \cdots & 0 & 0 & \chi_{\beta} & \mathcal{N}(0,2)
\end{bmatrix},
\]
whose eigenvalues have joint density
\[
 \frac{1}{Z_{n,\beta}}
 \prod_{1\le i<j\le n}|x_i-x_j|^\beta
 \exp\!\left(-\frac12\sum_{i=1}^n x_i^2\right).
\]
This Gibbs form connects the model to logarithmic gases and makes many
global and fluctuation questions accessible; see, for example, \cite{S24}.

The parameter $\beta>0$ extends the three classical symmetry classes
$\beta=1,2,4$.  Forrester \cite{forrester2010log} gives a comprehensive
account of the connection between random matrices and log-gases.  Recent
work includes combinatorial investigations of the Dumitriu--Edelman model
\cite{B25} and large deviations for the largest eigenvalue in broad classes
of tridiagonal models \cite{GM25}.  Ram\'irez, Rider, and Vir\'ag
\cite{ramirez2011beta} identified the soft-edge continuum limit with the
Stochastic Airy Operator.  CMV models provide related extensions for
unitary ensembles \cite{killip2009eigenvalue}; see also
\cite{krishnapur2016universality}.

General tridiagonal models, including $\beta$-ensembles, were studied by
combinatorial methods in \cite{P09}.  Related developments include
\cite{augeri2023clt,liu2012fluctuations,trinh2019global,bose2022xx,
basak2011limiting,bose2010patterned,bose2021some,chafai2009aspects,
li2015fluctuations,bose2013finite}.

Tridiagonalizing a Wigner ensemble produces dependent entries.  Here we
first study models with i.i.d.\ off-diagonal entries, for which the local
resolvent recursion remains explicit.  These models also connect to band
matrices whose width grows with the matrix size, where both eigenvalue and
eigenvector questions remain subtle.

There is also a long physics tradition in which tridiagonal, or Jacobi, random matrices appear as one-dimensional tight-binding Hamiltonians.  In the Anderson model, diagonal disorder and nearest-neighbor hopping lead to questions about the density of states and localization of eigenvectors \cite{Anderson1958,Lagendijk2009}.  Classical works of Dyson on disordered oscillator chains and Halperin on one-dimensional random potentials already describe spectral measures through recursions or integral equations \cite{Dyson1953,Halperin1965}; the relation between the integrated density of states and the Lyapunov exponent, or localization length, is due to Thouless \cite{Thouless1972}.  Later developments used supersymmetric transfer matrices and related cavity equations for one-dimensional Anderson and tree-like models \cite{CampaninoKlein1986,Efetov1983,AbouChacra1973,DerridaRodgers1993}.  Recently, Evnin gave an explicit supersymmetric-transfer-operator construction for the density of states of the one-dimensional Anderson model with arbitrary on-site and hopping distributions \cite{Ev25}.  Our approach is different in spirit: it stays within a probabilistic Stieltjes-transform and continued-fraction framework and is aimed at limiting empirical measures of the tridiagonal models below, but the underlying random recursion for diagonal resolvent entries is closely related to this physics literature.
\medskip

The tridiagonal model considered in \cite{P09} is
\begin{equation}\label{mod1}
    A^\alpha_N = \begin{bmatrix}
    a_N & b_{N-1} & 0 & \cdots & 0 \\
    {b}_{N-1} & a_{N-1} & b_{N-2} & \ddots & \vdots \\
    0 & {b}_{N-2} & a_{N-2} & \ddots & 0 \\
    \vdots & \ddots & \ddots & \ddots & b_{1} \\
    0 & \cdots & 0 & {b}_{1} & a_1
\end{bmatrix},
\end{equation}
where the parameter $\alpha > 0$ is related to the non-degenerate convergence of the moments
\begin{align}\label{moment}
    \forall k : \lim_{n\to +\infty} \mathbb{E}\left(\left(\frac{b_n}{n^\alpha}\right)^k\right) = m_k < +\infty.
\end{align}
All entries are real.  In the result recalled below, the diagonal variables
have moments of every order, uniformly in the index.
\medskip

The rescaled matrix $X^\alpha_N=N^{-\alpha}A^\alpha_N$ provides a suitable
framework for the spectral distribution.  The method of moments studies
$L_k=\trN((X_N^\alpha)^k)$, where $\operatorname{Tr}$ denotes the ordinary
matrix trace and $\trN=N^{-1}\operatorname{Tr}$ denotes the normalized
trace.  Under the preceding assumptions, the main
theorem of \cite{P09} gives almost-sure convergence of all these moments.
For ease of comparison with that source, the five displays in the recalled
statement retain their original equation numbers.

\begin{theorem}[Theorem 1 in \cite{P09}]\label{T1}
Let $\alpha > 0$. Assume that all random variables $a_n$ and $b_n$ are independent and there exists a sequence $\{m_k\}_{k \geq 0}$ with $m_0 = 1$ such that
\begin{equation}
    \lim_{n \to \infty} \mathbb{E} \left( \frac{b_n}{n^\alpha} \right)^k = m_k \quad \text{for any } k \geq 0. \tag{2.7}
\end{equation}
Moreover, assume
\begin{equation}
    \sup_{n \geq 1} \mathbb{E} |a_n|^k < \infty \quad \text{for any } k \geq 0. \tag{2.8}
\end{equation}
Denoting $X_n = \frac{1}{n^\alpha} A_n$, we have
\begin{equation}
    \lim_{n \to \infty} \mathbb{E} \, \mathrm{tr}_n(X_n^k) = L_k \quad \text{for any } k \geq 0, \tag{2.9}
\end{equation}
and almost surely,
\begin{equation}
    \lim_{n \to \infty} \mathrm{tr}_n(X_n^k) = L_k \quad \text{for any } k \geq 0. \tag{2.10}
\end{equation}
Moreover, $L_k$ is given by
\begin{equation}
    L_k =
    \begin{cases}
        0 & \text{if } k \text{ is odd}, \\
        \frac{1}{\alpha k + 1} \sum_{\gamma \in \Gamma_k} \prod_{i <0} m_{l_i(\gamma)} & \text{if } k \text{ is even},
    \end{cases}
    \tag{2.11}
\end{equation}
where $\Gamma_k$ is the set of closed nearest-neighbour paths of length
$k$ whose maximum is $0$:
\[
\Gamma_k = \left\{(\gamma_0,\ldots,\gamma_k)\in\mathbb Z^{k+1}:
\gamma_0=\gamma_k,\ \max_j\gamma_j=0,\
|\gamma_j-\gamma_{j-1}|=1\right\},
\]
and
\[
l_i(\gamma):=\#\{1\le j\le k:
 \{\gamma_{j-1},\gamma_j\}=\{i,i+1\}\}.
\]
Notably, $L_k = 0$ for all odd $k$.
\end{theorem}

As noted in \cite[Section~1]{AWZ06}, if the limiting moment sequence
determines a probability measure, the moment convergence implies almost-sure
weak convergence of the empirical spectral distribution.  Moment
determinacy is automatic for compactly supported limits.
\medskip

The same path expansion describes joint moments of several tridiagonal
matrices through colored nearest-neighbour paths; see Section~\ref{s:6}.
\medskip

Our first goal is to extend Theorem~\ref{T1} by a Stieltjes-transform
argument that requires only a second moment of the off-diagonal entries.
Our second goal is to identify nearly optimal deterministic conditions under
which an inhomogeneous profile produces a scale mixture of the simple limit.
\medskip 

The Stieltjes transform is the natural analytic tool for both purposes.
\bigskip

We first fix the notation used below.

\subsection{Preliminaries}

\begin{definition}[Norms, traces, and empirical measures]
For $f:\R\to\R$, set
\[
 [f]_{\mathrm{Lip}}:=\sup_{x\ne y}\frac{|f(x)-f(y)|}{|x-y|},
 \qquad
 \|f\|_{\mathrm{BL}}:=\|f\|_\infty+[f]_{\mathrm{Lip}}.
\]
For $A\in M_N(\C)$, write
\[
 \|A\|_{\mathrm{HS}}:=\bigl(\mathrm{Tr}(A^*A)\bigr)^{1/2},
 \qquad
 \|A\|_{2,N}:=N^{-1/2}\|A\|_{\mathrm{HS}},
 \qquad
 \trN(A):=N^{-1}\mathrm{Tr}(A).
\]
Here $\mathrm{Tr}$ is the ordinary, unnormalized matrix trace.  If
$A\in\mathcal H_N(\C)$ has eigenvalues
\[
 \lambda_1(A),\ldots,\lambda_N(A),
\]
its empirical spectral distribution (ESD) is
\[
 \mu_A:=\frac1N\sum_{i=1}^N\delta_{\lambda_i(A)}.
\]
\end{definition}

Write $\C^+=\{z\in\C:\Im z>0\}$ and $\C^-=\{z\in\C:\Im z<0\}$.
For a random variable $Y$, the notation $\mathcal L(Y)$ denotes its law,
and $\stackrel d=$ means equality in distribution.  If $\mu$ has a finite
$q$th moment, we write
\[
 m_q(\mu):=\int_{\R}x^q\,d\mu(x).
\]

\begin{definition}[Stieltjes transform]
For a finite measure $\mu$ on $\R$, define
\[
 S_\mu(z):=\int_\R\frac{1}{z-x}\,d\mu(x),\qquad z\in\C^+.
\]
If $M\in M_N(\C)$ is Hermitian, we abbreviate
\[
 S_M(z):=\trN\bigl((zI_N-M)^{-1}\bigr).
\]
\end{definition}

The map $S_\mu$ sends $\C^+$ into $\C^-$.  Stieltjes transforms
characterize weak convergence: it is enough to verify convergence on a
countable subset of $\C^+$ with an accumulation point; see
\cite[Theorem~2.4.4]{AGZ10}.

On $\C\setminus[-r,r]$, we use the analytic branch of the square root for
which $\sqrt{z^2-r^2}\sim z$ at infinity.

\begin{example}\label{semi}
The semicircle law $\sigma_r$ on $[-r,r]$ has density
\(
 \ind_{[-r,r]}(x)\frac{2}{\pi r^2}\sqrt{r^2-x^2}
\)
and Stieltjes transform
\[
 S_{\sigma_r}(z)=\frac{2(z-\sqrt{z^2-r^2})}{r^2}
 =\frac{1}{z-(r/2)^2S_{\sigma_r}(z)}.
\]
\end{example}

\begin{example}
The arcsine law $\gamma_r$ on $[-r,r]$ has density
\(
 \ind_{[-r,r]}(x)/(\pi\sqrt{r^2-x^2})
\)
and Stieltjes transform
\[
 S_{\gamma_r}(z)=\frac{1}{\sqrt{z^2-r^2}}.
\]
\end{example}

For $p\ge1$, let $\mathcal P_p(\R)$ be the probability measures with finite
$p$th moment.  For $\mu,\nu\in\mathcal P_p(\R)$, set
\[
 W_p(\mu,\nu):=
 \inf_{\pi\in\Pi(\mu,\nu)}
 \left(\int_{\R^2}|x-y|^p\,d\pi(x,y)\right)^{1/p}.
\]
Here $\Pi(\mu,\nu)$ denotes the set of couplings of $\mu$ and $\nu$.
The space $(\mathcal P_p(\R),W_p)$ is Polish, and $W_p$ convergence is
equivalent to weak convergence together with uniform integrability of the
$p$th powers; see \cite[Definition~6.8 and Theorem~6.9]{V09}.
We use the analogous notation $\mathcal P_1(E)$ when $E$ is a closed subset
of $\C$, equipped with the Euclidean metric.

We now state the assumptions for the main theorem.

\subsection{Main result}
We now extend $X_N^\alpha$ to a deterministic triangular array of real
numbers
$(\sigma_{k,N})_{1\le k\le N,\,N\ge1}$.  Define
\begin{equation}\label{model}
X_N^\sigma=\begin{bmatrix}
a_{N,N} & \sigma_{N-1,N} b_{N-1} & 0 & \cdots & 0\\
\sigma_{N-1,N}{b}_{N-1} & a_{N-1,N} & \sigma_{N-2,N} b_{N-2} & \ddots & \vdots\\
0 & \sigma_{N-2,N}{b}_{N-2} & a_{N-2,N} & \ddots & 0\\
\vdots & \ddots & \ddots & \ddots & \sigma_{1,N} b_1\\
0 & \cdots & 0 & \sigma_{1,N}{b}_1 & a_{1,N}
\end{bmatrix}
\end{equation}
The case $X_N^\alpha$ is recovered when the sequence $\sigma$ is $\left(\frac{k^\alpha}{N^\alpha}\right)_{k\leq N,\, N\in\mathbb{N}}$.
\medskip 

When $\sigma\equiv1$, we write $X_N$ and call it the simple tridiagonal
model; see Section~\ref{s:2}.

\begin{assumptions}\label{assumptionmatrix}
The random entries satisfy
\begin{itemize}
\item[(H1)] $(b_k)_{k\ge1}$ are i.i.d. real random variables with
  $\E b_1^2<\infty$;
\item[(H2)] the real diagonal entries form an independent family, are
  independent of the off-diagonal entries, and satisfy
  \[
    \frac1N\sum_{k=1}^N\E|a_{k,N}|^2\longrightarrow0.
  \]
\end{itemize}
\end{assumptions}

For each $N$, we regard $(\sigma_{k,N})_{1\le k\le N}$ as a
deterministic row of real numbers.  The last entry is immaterial for the
$N\times N$ matrix, but it permits us to formulate the local condition without
an endpoint convention.  Put
\[
 \overline\nu_N:=\frac1{N-1}\sum_{k=1}^{N-1}\delta_{\sigma_{k,N}},
 \qquad
 \eta_N:=\frac1{N-1}\sum_{k=1}^{N-1}
          \delta_{(\sigma_{k,N},\sigma_{k+1,N})},
\]
and let $\Delta(t):=(t,t)$.

\begin{assumption}[Local profile convergence]\label{assumptionsigma}
There is a probability measure $\mu_T$ on $\R$ such that
\begin{equation}\tag{H3}\label{H3}
 \eta_N\Rightarrow\Delta_\#\mu_T.
\end{equation}
Equivalently, the empirical law of two consecutive deformation coefficients
converges to the law of $(T,T)$, where $T\sim\mu_T$.
\end{assumption}

\begin{remark}[Stronger, practical hypotheses]\label{rem:strong-sigma}
Taking the first marginal in \eqref{H3} gives automatically
$\overline\nu_N\Rightarrow\mu_T$.  The empirical measure formed with all
$N$ coefficients has the same weak limit, since the two measures differ only
by endpoint masses of order $N^{-1}$.  Thus no separate marginal assumption
is needed.

Two useful supplementary conditions will be used below.  The first is
understood when $\mu_T\in\mathcal P_2(\mathbb R)$:
\begin{align}
 W_2(\overline\nu_N,\mu_T)&\longrightarrow0,
 &\tag{H3$_2$}\label{H3-2}\\
 \max_{1\le k<N}|\sigma_{k+1,N}-\sigma_{k,N}|&\longrightarrow0.
 &\tag{H3$_\infty$}\label{H3-infty}
\end{align}
Condition \eqref{H3-2} is the second-moment strengthening of the marginal
consequence of \eqref{H3}.  If $\overline\nu_N\Rightarrow\mu_T$, then
\eqref{H3-infty} implies \eqref{H3}; however, the converse is false because
\eqref{H3} permits a vanishing proportion of macroscopic jumps.  More
precisely, \eqref{H3} is equivalent to marginal convergence together with
the averaged oscillation condition
\[
 \frac1{N-1}\sum_{k=1}^{N-1}
 \bigl(1\wedge|\sigma_{k+1,N}-\sigma_{k,N}|\bigr)\longrightarrow0.
\]
Thus the pair formulation is the natural relaxed version of a vanishing
mesh for the finite-window path argument.

Equivalently, in the presence of the marginal weak convergence supplied by
\eqref{H3}, condition \eqref{H3-2} is the asymptotic
uniform-integrability requirement
\[
 \lim_{R\to\infty}\limsup_{N\to\infty}
 \frac1{N-1}\sum_{k=1}^{N-1}
 \sigma_{k,N}^2\ind_{\{|\sigma_{k,N}|>R\}}=0.
\]
\end{remark}

\begin{remark}[A further sign-invariant relaxation]\label{rem:absolute-sigma}
For real symmetric tridiagonal matrices, diagonal conjugation by signs shows
that $X_N^\sigma$ and $X_N^{|\sigma|}$ have the same eigenvalues.  Thus, for
the spectral conclusion alone, it is enough that the adjacent-pair measure of
$|\sigma_{k,N}|$ converge to the law of $(R,R)$ for some $R\ge0$; its
marginal then converges automatically to the law of $R$.  The limit is
then
$\mathcal L(RX)$.  This is strictly weaker when signs oscillate, and it agrees
with $\mathcal L(TX)$ whenever $R\stackrel d=|T|$, because $\mu_b$ is
symmetric.
\end{remark}

\begin{theorem}\label{main}
Let $(X^\sigma_N)$ be the tridiagonal model \eqref{model} and $(X_N)$ the associated simple tridiagonal model.
Under Assumptions~\ref{assumptionmatrix} and condition~\eqref{H3}, the
empirical spectral distribution of $X_N^\sigma$ converges weakly almost surely
to
\[
 \mu_{T,b}:=\mathcal L(TX),
\]
where $T\sim\mu_T$, $X\sim\mu_b$, and $T$ and $X$ are independent.  In
particular,
\[
 S_{\mu_{T,b}}(z)=\int_{\R^2}\frac{1}{z-tx}\,d\mu_T(t)\,d\mu_b(x).
\]
If, in addition, $T\in L^2$ and \eqref{H3-2} holds, then
\[
 W_2(\mu_{X_N^\sigma},\mu_{T,b})\longrightarrow0
 \qquad\text{in probability}.
\]
For $\Im z>\sqrt{\E b_1^2}$, the law $\mu_b$ is characterized by
\[
S_{\mu_b}(z)=\E\!\left[
 \frac{1}{z-b_1^2S_1(z)-b_2^2S_2(z)}
\right],
\]
where $b_1,b_2$ are i.i.d. copies of a variable $b$ with the common law of
the off-diagonal entries, $S_1,S_2$ are i.i.d., and all four variables are
independent.  Moreover,
\[
 S_j(z)\stackrel d=\frac{1}{z-b^2S_j(z)},\qquad j=1,2,
\]
where, in each equation, $b$ denotes a fresh copy of the off-diagonal
variable, independent of $S_j(z)$.
On the remainder of $\C^+$, $S_{\mu_b}$ is the unique analytic
continuation of this expression.
\end{theorem}

Section~\ref{s:6} identifies the common-shift operator structure governing
joint moments of several tridiagonal models and explains why these limits
differ from Voiculescu asymptotic freeness.

\paragraph{\textbf{Structure of the paper.}}

In Section~\ref{s:2} we prove convergence for the simple tridiagonal model
without a deformation; see Theorem~\ref{T2}.

In Section~\ref{s:3} we introduce the deformation $\sigma$, isolate the
local profile condition~\eqref{H3}, and prove the general
scale-mixture result in Theorem~\ref{T3}.

Section~\ref{s:5} gives examples illustrating the range of limiting laws.

In Section~\ref{s:4} we prove joint moment convergence of the deterministic
profile and the simple model.  Section~\ref{s:6} gives colored-path joint
moments for several independent tridiagonal matrices.  We then formulate
the deformation through the profile--shift pair and prove that convergence
of its joint $^*$-distribution suffices for the bounded zero-diagonal
model; a diagonal satisfying (H2) can then be restored by
Lemma~\ref{deform}.

\section{The simple tridiagonal model}\label{s:2}
We first study the simple model with independent entries and no deterministic
deformation.  The deformation is treated in Section~\ref{s:3}.
\medskip

The \emph{simple tridiagonal model} is the model with $\sigma\equiv1$ and
zero diagonal entries:
\begin{equation}\label{X_N}
    X_N = \begin{bmatrix}
    0 & b_{N-1} & 0 & \cdots & 0 \\
    {b}_{N-1} & 0 & b_{N-2} & \ddots & \vdots \\
    0 & {b}_{N-2} & 0 & \ddots & 0 \\
    \vdots & \ddots & \ddots & \ddots & b_{1} \\
    0 & \cdots & 0 & {b}_{1} & 0
    \end{bmatrix}
\end{equation}
where $(b_i)_{i\ge1}$ are real i.i.d.\ random variables in $L^2$.
For the remainder of this section we conjugate by the reversal permutation
and retain the notation $X_N$.  Thus $b_k$ lies on the edge
$\{k,k+1\}$.  Since the entries are i.i.d., this matrix has the same law
as \eqref{X_N}.
\medskip

We recall a useful lemma for the rest of the paper, which reduces the study to this case.

\begin{lemma}[Hoffman--Wielandt]\label{HW}
Let $A,B \in \mathcal{H}_N(\mathbb{C})$, and let $\lambda_1^A \le \dots \le \lambda_N^A$ and $\lambda_1^B \le \dots \le \lambda_N^B$ be their eigenvalues. Then
\[
 \sum_{i=1}^N |\lambda_i^A-\lambda_i^B|^2
 \le \|A-B\|_{\mathrm{HS}}^2.
\]
\end{lemma}
\medskip

To describe the limiting Stieltjes transform, write, for $z\in\C^+$,
\[
G_N(z):=(zI_N-X_N)^{-1}
\]
for the resolvent matrix; its normalized trace is $S_{X_N}(z)$.
The starting point is the following lemma.

\begin{lemma}[Lemma 2.4.6 in 
\cite{AGZ10}]\label{rec}
Let $A\in\mathcal H_N(\C)$.  For each $i$, let $v_i$ be the $i$th
column of $A$ with its $i$th entry removed, so that
$v_i\in\C^{N-1}$, and let $A^{(i)}\in M_{N-1}(\C)$ be the matrix obtained
by deleting row and column $i$.  Then, for $z\in\C^+$,
\[
(zI_N-A)^{-1}(i,i)
=\frac{1}{z-A(i,i)-v_i^*(zI_{N-1}-A^{(i)})^{-1}v_i}.
\]
\end{lemma}

Let $X_N^{(i)}$ be obtained from $X_N$ by deleting row and column $i$, and
let $x_i$ be the corresponding column with its $i$th entry removed.
Lemma~\ref{rec} gives
\begin{align*}
S_{X_N}(z)
&=\frac1N\sum_{i=1}^N
 \frac{1}{z-X_N(i,i)-x_i^*(zI_{N-1}-X_N^{(i)})^{-1}x_i}\\
&=\frac1N\sum_{i=1}^N
 \frac{1}{z-x_i^*(zI_{N-1}-X_N^{(i)})^{-1}x_i}.
\end{align*}
The deletion separates the path into two independent blocks:
\[
X_N^{(i)}=\begin{bmatrix}
    \tilde{X}_{i-1}& 0 \\
    0 & \tilde{X}_{N-i}
\end{bmatrix},
\]
where $\tilde{X}_{i-1}$ and $\tilde{X}_{N-i}$ have the same distribution as $X_{i-1}$ and $X_{N-i}$. Thus
\[
(zI_{N-1}-X_N^{(i)})^{-1}=
\begin{bmatrix}
    \widetilde G_{i-1}(z)&0\\
    0&\widetilde G_{N-i}(z)
\end{bmatrix},
\]
where the two resolvents are independent copies of the corresponding
smaller resolvents.  Consequently,
\begin{equation}\label{cle}
 x_i^*(zI_{N-1}-X_N^{(i)})^{-1}x_i
 =b_{i-1}^2\widetilde G_{i-1}(z)(i-1,i-1)
  +b_i^2\widetilde G_{N-i}(z)(1,1).
\end{equation}

We first identify the endpoint resolvent entries.

\begin{proposition}
For every $z\in\C^+$ and $i\ge1$, the variables $G_i(z)(i,i)$ and
$G_i(z)(1,1)$ have the same distribution.
\end{proposition}

\begin{proof}
Fix $z\in\C^+$ and set $P_i(z):=\det(zI_i-X_i)$.  Expansion along the
first row gives
\begin{equation}\label{eq:P}
P_{i+1}(z) = zP_i(z) -  b_i^2 P_{i-1}(z).
\end{equation}
The adjugate formula gives
\[
(zI_i-X_{i})^{-1}(1,1)=\frac{\det(zI_{i-1}-X_{i-1})}{\det(zI_i-X_{i})}=\frac{P_{i-1}(z)}{P_{i}(z)}.
\]
Thus, with $s_i:=G_i(z)(1,1)$, Equation~\eqref{eq:P} yields
\[
s_{i+1} = \frac{1}{\,z - b_i^2 s_i\,}, \qquad s_1 = \frac{1}{z}.
\]

The reversal permutation takes $X_i$ to a matrix with the coefficients in
reverse order.  Since the coefficients are i.i.d., the reversed matrix has
the same law as $X_i$.  Its upper-left resolvent entry is the lower-right
resolvent entry of $X_i$, which proves the assertion.
\end{proof}

The endpoint symmetry and the Schur-complement identity give, after taking
expectations and with the boundary convention $b_0=b_N=0$,
\begin{equation}\label{Stil}
\E S_{X_N}(z)
=\frac1N\sum_{i=1}^N
\mathbb E\!\left[
 \frac{1}{z-b_{i-1}^2s_{i-1}-b_i^2\widetilde s_{N-i}}
\right],
\end{equation}
where, in each summand, $s_{i-1}$ and $\widetilde s_{N-i}$ are
independent endpoint fractions and are independent of the two adjacent
coefficients.  This is a termwise identity of expectations; it is not an
equality in distribution for the entire sum.  We set $s_0=0$ as the
zero-depth boundary value corresponding to the absence of an edge outside
the matrix.
\medskip

\begin{lemma}\label{Wass}
Let $(\beta_i)_{i\ge1}$ be i.i.d.\ copies of $b_1$, set $s_0=0$, and
define
\[
 s_i=\frac{1}{z-\beta_i^2s_{i-1}},\qquad i\ge1.
\]
If $\E b_1^2<\infty$ and $\Im z>\sqrt{\E b_1^2}$, then $s_i$ converges
in distribution to a random variable $S(z)$ satisfying
\[
 S(z)\stackrel d=\frac{1}{z-b^2S(z)},
\]
where $b$ is a copy of $b_1$ independent of $S(z)$.  Thus the law of
$S(z)$ is the limit of the laws of the finite continued fractions
\[
S(z)=\cfrac{1}{z-\cfrac{\; b_1^2\; }{\; z-\cfrac{\; b_2^2\; }{\dots}\; }\; \; }\,.
\]
\end{lemma}

\begin{proof}
Let $\Im z>\sqrt{\E b_1^2}$.  Then $s_i\in\C^-$ and
$|s_i|\le1/\Im z$.  Denote the laws of $s_n$ and $b_1$ by
$\mathbb P_n$ and $\mathbb P_b$, respectively, and let
$f$ be $1$-Lipschitz on $\C$.  In this proof, $W_1$ denotes the
$1$-Wasserstein distance on $\C\simeq\R^2$ for the Euclidean metric.
\begin{align*}
\mathbb{E}\big(f(s_{n+1})\big) - \mathbb{E}\big(f(s_n)\big)
&=\int_\mathbb{R} \Big[\mathbb{E}\!\left( f\!\left(\frac{1}{z-x^2 s_n}\right)\right) -  \mathbb{E}\!\left(f\!\left(\frac{1}{z-x^2 s_{n-1}}\right)\right)\Big]\, d\mathbb{P}_b(x)\\
& = \int_\mathbb{R} \Big[\mathbb{E}\big(f\circ g_x(s_n)\big) - \mathbb{E}\big(f\circ g_x(s_{n-1})\big)\Big] \,d\mathbb{P}_b(x),
\end{align*}
where $g_x(w)=1/(z-x^2 w)$.  The function $f\circ g_x$ is
$x^2/(\Im z)^2$-Lipschitz, so
\[
\left|\mathbb{E}\big(f\circ g_x(s_n)\big)
- \mathbb{E}\big(f\circ g_x(s_{n-1})\big)\right|
\le \frac{x^2}{(\Im(z))^2} W_1(\mathbb{P}_n,\mathbb{P}_{n-1}).
\]
Integrating yields
\[
W_1(\mathbb{P}_{n+1},\mathbb{P}_{n})\le \frac{\mathbb{E}(b^2)}{(\Im(z))^2}\,W_1(\mathbb{P}_n,\mathbb{P}_{n-1}).
\]
Thus $(\mathbb{P}_n)$ is Cauchy when
$\Im z>\sqrt{\E b_1^2}$ and hence convergent.  To identify its limit,
define, for $\nu\in\mathcal P_1(\overline{\C^-})$,
\[
 \Phi_z(\nu):=
 \mathcal L\!\left(\frac{1}{z-b^2Y}\right),
 \qquad Y\sim\nu,
\]
where $b$ and $Y$ are independent.  The recursion gives
$\mathbb P_{n+1}=\Phi_z(\mathbb P_n)$, and the estimate above shows that
$\Phi_z$ is a strict contraction in $W_1$.  If
$\mathbb P_n\to\mathbb P_\infty$, continuity of this contraction gives
$\mathbb P_\infty=\Phi_z(\mathbb P_\infty)$, which is precisely the
displayed distributional fixed-point equation.  The same contraction shows
that this invariant law is unique.
\end{proof}

We also recall the discrete Burkholder inequality.

\begin{lemma}[Burkholder inequality; see \cite{Burkholder1973}]
Let $Y\in L^p(\Omega,\mathcal A,\mathbb P)$, $p>2$, and let
$(\mathcal F_k)$ be a filtration.  Set
$Y_k:=\E(Y\mid\mathcal F_k)-\E(Y\mid\mathcal F_{k-1})$.  Then
\[
 \E\left|\sum_{k=1}^nY_k\right|^p
 \le C_p\E\left(\sum_{k=1}^n|Y_k|^2\right)^{p/2},
\]
where $C_p$ depends only on $p$.
\end{lemma}

We add a classical bound between resolvents of minors.

\begin{proposition}[Lemma~1.5.6 in \cite{C12}]
Let $M$ be an $n\times n$ Hermitian matrix and let $M_k$ be obtained by
deleting row and column $k$.  Then, for $z\in\C^+$,
\[
 \left|\mathrm{Tr}((zI_n-M)^{-1})
 -\mathrm{Tr}((zI_{n-1}-M_k)^{-1})\right|
 \le\frac{1}{\Im z}.
\]
\end{proposition}
This is the deletion estimate used in the concentration argument below.
Indeed, if one diagonal entry is resampled, deletion of its vertex gives a
common minor for the two matrices; if one off-diagonal entry is resampled,
deletion of its two incident vertices gives a common minor.  Applying the
proposition once or twice to each matrix and dividing by $N$ shows that the
corresponding normalized resolvent traces differ by at most
$C/(N\Im z)$.
\medskip

We shall also use the rank inequality: for Hermitian $A,B\in M_N(\C)$,
\begin{equation}\label{eq:rank}
 \sup_x|F_A(x)-F_B(x)|\le\frac{\operatorname{rank}(A-B)}N,
 \qquad
 |S_A(z)-S_B(z)|\le
 \frac{2\operatorname{rank}(A-B)}{N\Im z},
\end{equation}
where $F_A,F_B$ are the distribution functions of the two ESDs; see
\cite[Lemma~2.6]{bai1999}.

Combined with the common-minor estimate above, Burkholder's inequality gives
the following concentration statement.

\begin{corollary}
Let $X_N$ be a tridiagonal model whose entries are independent.  For every
$z\in\C^+$ and $p>2$,
\[
 \E\left|S_{X_N}(z)-\E S_{X_N}(z)\right|^p
 \le C_{p,z}N^{-p/2}.
\]
Consequently $S_{X_N}(z)-\E S_{X_N}(z)\to0$ almost surely.
\end{corollary}

Indeed, revealing one entry at a time and using \eqref{eq:rank} bounds
each martingale difference by $C_z/N$; equivalently, this follows from the
common-minor comparison immediately above.  Burkholder's inequality and
Borel--Cantelli give the claim.  We can therefore first identify the limit
of the expected Stieltjes transform.
\medskip 

The characterization of our principal result is naturally connected with several
classical descriptions of limiting spectral distributions in random matrix theory.
For Wigner and Wishart matrices, the limiting laws are the semicircle law
\cite{Wigner1958,AGZ10,C12} and the Mar\v{c}enko--Pastur law
\cite{MP67,AGZ10,C12}, respectively; in both cases the Stieltjes transform is
characterized by a self-consistent algebraic equation, obtained from resolvent
identities or equivalent orthogonal-polynomial/continued-fraction arguments.
For special tridiagonal and Jacobi matrix models, this point of view is especially
transparent: the three-term recurrence for orthogonal polynomials yields explicit
second-order equations for the corresponding continued fractions, recovering the
semicircle, Mar\v{c}enko--Pastur, and Kesten--McKay laws
\cite{Kesten1959,McKay1981,Dumitriu_2006,DubbsEdelman2015}.  The same
framework also includes Jacobi or MANOVA ensembles: the beta-Jacobi tridiagonal
models of Killip--Nenciu and Edelman--Sutton \cite{KN04,ES08} are connected
to MANOVA statistics, whose limiting spectral distribution is the Wachter law
\cite{Wachter1980,Johnstone2008,EF13}.  Theorem~\ref{T2} fits into this literature
by giving, for a broad i.i.d. tridiagonal model under only a second-moment
assumption, a limiting measure characterized by a random continued-fraction
fixed point for its Stieltjes transform, rather than by a deterministic
finite-degree algebraic equation.

\begin{theorem}\label{T2}
Let $(X_N)$ be the simple tridiagonal model defined above. If $(b_i)$ is a
sequence of i.i.d.\ variables with a finite second moment, then the
empirical spectral distribution of $X_N$ converges almost surely to
$\mu_b$.  For $\Im z>\sqrt{\E b_1^2}$, its Stieltjes transform is
characterized by
\[
S_{\mu_b}(z) = \mathbb{E}\!\left(\frac{1}{\,z- b_1^2 S_1(z)- b_2^2 S_2(z)\,}\right),
\]
where $S_1(z),S_2(z)$ are i.i.d.\ copies of the fixed-point variable
in Lemma~\ref{Wass}, $b_1,b_2$ are i.i.d.\ copies of an off-diagonal
entry, and all four variables are independent.
On the remainder of $\C^+$, $S_{\mu_b}$ is the unique analytic
continuation of this expression.
\end{theorem}

\begin{proof}
Fix $z$ with $\Im z>\sqrt{\E b_1^2}$.  In \eqref{Stil}, all but
$2N_0$ summands have both endpoint fractions of depth at least $N_0$.
Lemma~\ref{Wass} and the resolvent Lipschitz estimate used in its proof
therefore imply
\[
 \E S_{X_N}(z)\longrightarrow
 \E\!\left[\frac{1}{z-b_1^2S_1(z)-b_2^2S_2(z)}\right].
\]

The expected ESDs are tight because
\[
 \int x^2\,d\E\mu_{X_N}(x)
 =\frac1N\E\mathrm{Tr}(X_N^2)
 =\frac{2(N-1)}N\E b_1^2.
\]
Their Stieltjes transforms form a normal family on $\C^+$.  Every
subsequential analytic limit agrees with the preceding expression on the
open set $\{\Im z>\sqrt{\E b_1^2}\}$; Montel's theorem and the identity
theorem therefore give convergence on all of $\C^+$ to the Stieltjes
transform of a probability measure $\mu_b$.

The martingale concentration estimate above gives almost-sure convergence
of Stieltjes transforms on a countable dense subset of $\C^+$.  Finally,
$N^{-1}\mathrm{Tr}(X_N^2)\to2\E b_1^2$ almost surely by the strong law, so
the random ESDs are tight.  The Stieltjes criterion completes the proof.
\end{proof}

\begin{proposition}[Second moment of the simple limit]
\label{prop:simple-W2}
Under the hypotheses of Theorem~\ref{T2}, the limiting law belongs to
$\mathcal P_2(\mathbb R)$ and
\[
 m_2(\mu_b)=2\E b_1^2.
\]
In fact,
\[
 W_2(\mu_{X_N},\mu_b)\longrightarrow0
 \qquad\text{almost surely}.
\]
\end{proposition}

\begin{proof}
For $C>0$, let $b_k^{(C)}:=(-C)\vee(b_k\wedge C)$ and let $X_N^{(C)}$
be the simple model formed with these clipped entries.  Denote its limiting
law by $\mu_{b,C}$.  Theorem~\ref{T2} applies to $X_N^{(C)}$, and all its
eigenvalues lie in $[-2C,2C]$.  Hence weak convergence and the strong law
give
\[
 m_2(\mu_{b,C})
 =\lim_{N\to\infty}\frac1N\operatorname{Tr}((X_N^{(C)})^2)
 =2\E[(b_1^{(C)})^2].
\]

Couple $X_N$ and $X_N^{(C)}$ by using the same entries.  The
Hoffman--Wielandt inequality yields
\[
 W_2^2(\mu_{X_N},\mu_{X_N^{(C)}})
 \le\frac1N\operatorname{Tr}\bigl((X_N-X_N^{(C)})^2\bigr)
 =\frac2N\sum_{k=1}^{N-1}|b_k-b_k^{(C)}|^2.
\]
For integer $C$, all the relevant weak convergences and strong laws hold on
one common probability-one event.  Lower semicontinuity of $W_2$ under weak
convergence therefore gives
\[
 W_2^2(\mu_b,\mu_{b,C})
 \le2\E|b_1-b_1^{(C)}|^2\longrightarrow0.
\]
Consequently $\mu_b\in\mathcal P_2(\mathbb R)$ and, by continuity of second
moments under $W_2$ convergence,
\[
 m_2(\mu_b)=\lim_{C\to\infty}m_2(\mu_{b,C})=2\E b_1^2.
\]
Finally,
$N^{-1}\operatorname{Tr}(X_N^2)
=2N^{-1}\sum_{k<N}b_k^2\to2\E b_1^2$ almost surely.  Combining this
with Theorem~\ref{T2} and the characterization of $W_2$ convergence proves
the last assertion.
\end{proof}

Theorem~\ref{T2} also yields a perturbative extension to independent,
non-identically distributed off-diagonal entries whose laws converge in
$W_2$.  This situation occurs, for example, in normalized tridiagonal
$\beta$-ensemble models.  We first record the coupling fact that is needed.

\begin{lemma}\label{e:W_2err}
Let $(b_n)_{n\geq 1}$ be independent $\mathbb{R}^d$-valued random
variables, and let
\[
\mu_n:=\mathcal{L}(b_n), \qquad \mu:=\mathcal{L}(b),
\]
with $\mu_n,\mu\in\mathcal{P}_2(\mathbb{R}^d)$. Then the following are
equivalent:
\begin{enumerate}
    \item
    \[
    W_2(\mu_n,\mu)\longrightarrow 0.
    \]
    \item On some probability space, there exist random variables
    $(\widetilde b_n,\beta_n,\varepsilon_n)_{n\geq 1}$ such that
    \[
    (\widetilde b_n)_{n\geq 1}
    \stackrel{d}{=}
    (b_n)_{n\geq 1},
    \qquad
    (\beta_n)_{n\geq 1}\ \text{are i.i.d.\ with law }\mu,
    \]
    and
    \[
    \widetilde b_n=\beta_n+\varepsilon_n,
    \qquad
    \mathbb{E}|\varepsilon_n|^2\longrightarrow 0.
    \]
\end{enumerate}
Moreover, in the implication $(1)\Rightarrow(2)$, one may arrange that
\[
\mathbb{E}|\varepsilon_n|^2=W_2^2(\mu_n,\mu)
\qquad\text{for every }n.
\]
\end{lemma}

\begin{proof}
Assume $(1)$. For each $n$, let
$\pi_n\in\Pi(\mu,\mu_n)$ be an optimal coupling, and take the product
probability space with law $\bigotimes_{n\geq 1}\pi_n$. If
$(\beta_n,\widetilde b_n)$ denotes the $n$-th coordinate pair, then
the pairs are independent,
\[
\mathcal{L}(\beta_n)=\mu,
\qquad
\mathcal{L}(\widetilde b_n)=\mu_n,
\]
and therefore
\[
(\widetilde b_n)_{n\geq 1}
\stackrel{d}{=}
(b_n)_{n\geq 1}.
\]
Setting $\varepsilon_n:=\widetilde b_n-\beta_n$, we obtain
\[
\mathbb{E}|\varepsilon_n|^2
=
W_2^2(\mu_n,\mu)
\longrightarrow 0.
\]

Conversely, under $(2)$, the pair $(\widetilde b_n,\beta_n)$ is a
coupling of $\mu_n$ and $\mu$, so
\[
W_2^2(\mu_n,\mu)
\leq
\mathbb{E}|\widetilde b_n-\beta_n|^2
=
\mathbb{E}|\varepsilon_n|^2
\longrightarrow 0.\qedhere
\]
\end{proof}

Finally, we have this concluding result. 

\begin{corollary}
Let $X_N$ be the zero-diagonal simple model with independent, not
necessarily identically distributed, entries $b_n$.  If
$W_2(\mathcal L(b_n),\mathcal L(b))\to0$, then
\[
 W_2(\mu_{X_N},\mu_b)\longrightarrow0
 \qquad\text{in probability},
\]
where $\mu_b$ is the law in Theorem~\ref{T2}.
\end{corollary}

\begin{proof}
On the coupling supplied by Lemma~\ref{e:W_2err}, write
$b_n=\beta_n+\varepsilon_n$, where the $\beta_n$ are i.i.d.\ in $L^2$
and $\E|\varepsilon_n|^2\to0$.  If $Y_N$ is the simple model with
entries $\beta_n$, then
\[
\mathbb{E}\|X_N-Y_N\|_{\mathrm{HS}}^2
=2\sum_{i=1}^{N-1}\mathbb E|\varepsilon_i|^2=o(N).
\]
Hoffman--Wielandt therefore gives
$W_2(\mu_{X_N},\mu_{Y_N})\to0$ in probability.  Apply
Proposition~\ref{prop:simple-W2} to $Y_N$ and use the triangle inequality.
\end{proof}

\begin{figure}[ht!]
\centering
\includegraphics[width=8cm]{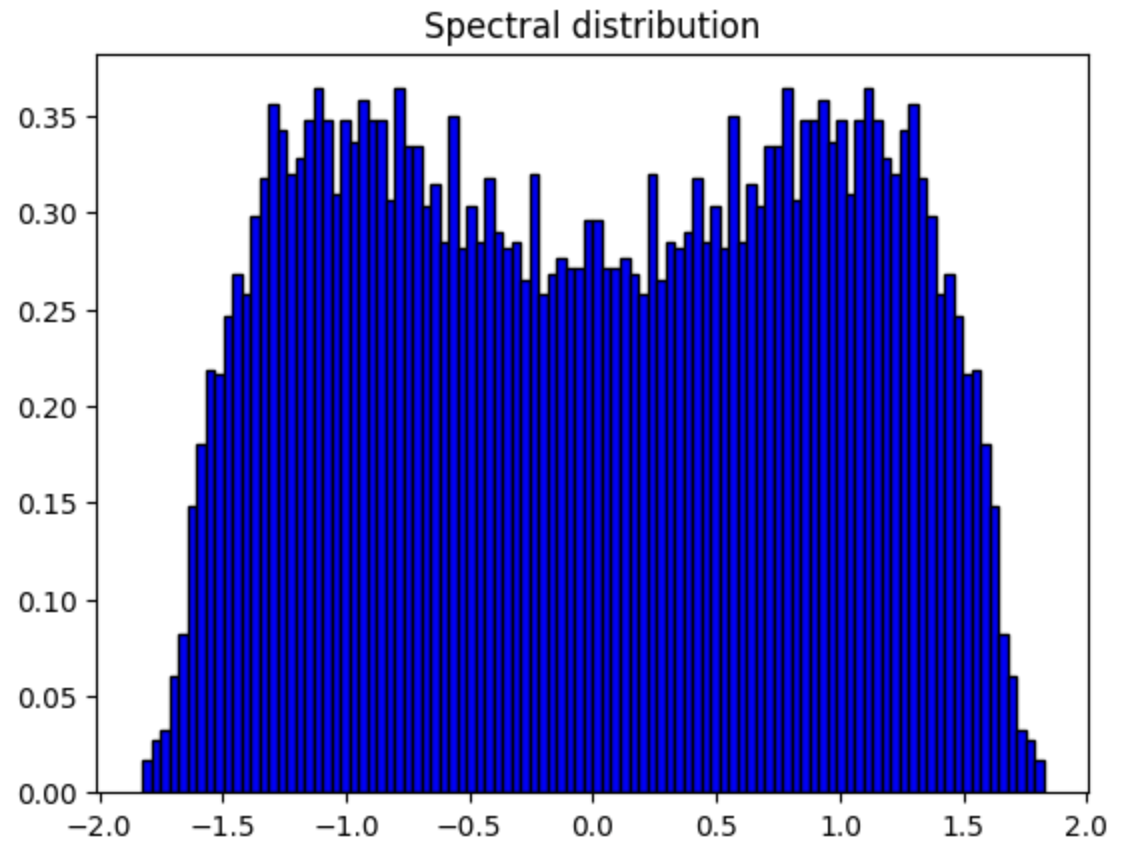}
\caption{Normalized eigenvalue histogram from one zero-diagonal simple
model of size $N=10^4$ with i.i.d. Pareto$(1,4)$ off-diagonal entries.  The
coefficient law has a finite second moment, as required in
Theorem~\ref{T2}, but is heavy tailed.}
\label{fig:pareto-simple}
\end{figure}

Figure~\ref{fig:pareto-simple} illustrates two features of the simple
model.  Although the coefficients are positive and strongly non-Gaussian,
the spectral distribution is symmetric: conjugation by the diagonal sign
matrix $\operatorname{diag}(1,-1,1,-1,\ldots)$ sends a zero-diagonal
tridiagonal matrix to its negative.  At the same time, the visible shape is
not the semicircle law.  It reflects the one-dimensional local recursion
in Theorem~\ref{T2} and retains information about the coefficient law.
This is the reference distribution that the deterministic profile will
rescale in the next section.

\section{The general tridiagonal model}\label{s:3}

We now prove the scale-mixture principle.  We first treat bounded
coefficients by the path method and then remove boundedness by a rank
truncation argument.

\subsection{Compact case}
We begin with the local consequence of \eqref{H3} which will be used in
every path expansion below.

\begin{lemma}[Finite-window convergence]\label{lem:local-block}
For every fixed $r\ge0$,
\[
 \frac1{N-2r}\sum_{p=r+1}^{N-r}
 \delta_{(\sigma_{p-r,N},\ldots,\sigma_{p+r,N})}
 \Rightarrow \mathcal L(T,\ldots,T).
\]
\end{lemma}

\begin{proof}
Let $I_{N,r}$ be uniformly distributed on
$\{r+1,\ldots,N-r\}$.  By \eqref{H3} and the Portmanteau theorem, for
every $\varepsilon>0$,
\[
 \frac1{N-1}\#\{k<N:
 |\sigma_{k+1,N}-\sigma_{k,N}|>\varepsilon\}\longrightarrow0.
\]
Translating the index by any fixed $j$ changes at most $O_r(1)$ terms.
Consequently, a union bound over the $2r$ edges in the window, followed by
the triangle inequality, gives
\[
 \max_{|j|\le r}
 |\sigma_{I_{N,r}+j,N}-\sigma_{I_{N,r},N}|
 \longrightarrow0
 \qquad\text{in probability}.
\]
The first marginal of \eqref{H3} also gives
$\sigma_{I_{N,r},N}\Rightarrow T$; deleting the $O_r(1)$ boundary
indices does not change this conclusion.  Slutsky's theorem now yields
\[
 (\sigma_{I_{N,r}-r,N},\ldots,\sigma_{I_{N,r}+r,N})
 \Rightarrow(T,\ldots,T),
\]
which is the claimed empirical convergence.
\end{proof}

For the proof of the compact theorem we use the following path notation.
For $q\ge0$, let
\[
 \mathcal C_q:=\{\gamma=(j_0,\ldots,j_q):
 j_0=j_q=0,\ |j_s-j_{s-1}|=1\text{ for }1\le s\le q\},
\]
and, for $i\in\mathbb Z$, let $\ell_i(\gamma)$ denote the number of
crossings of the edge $\{i,i+1\}$ by $\gamma$.  Finally, put
\[
 I_N(\gamma):=\{p:1\le p+j_s\le N\text{ for }0\le s\le q\}.
\]
Only the indices $|i|\le q$ can occur in $\ell_i(\gamma)$, and
$|I_N(\gamma)|=N+O_q(1)$.

\begin{theorem}[Scale-mixture principle: bounded case]
\label{thm:scale-mixture}
Let $(b_k)_{k\ge1}$ be i.i.d.\ real random variables with bounded
support, let $X_N$ be the associated zero-diagonal simple model, and let
$X_N^\sigma$ be its zero-diagonal deformation, whose $k$th off-diagonal
entry is $\sigma_{k,N}b_k$.  Suppose that the deterministic array $\sigma$
is uniformly bounded and satisfies \eqref{H3}.  Then
\[
 \mu_{X_N^\sigma}\Rightarrow\mathcal L(TX)
\]
almost surely, and the expected ESD converges weakly to the same law, where
$T\sim\mu_T$, $X\sim\mu_b$, and $T,X$ are independent.  Equivalently, for
every integer $q\ge0$,
\[
 m_q(\mathcal L(TX))=m_q(\mu_T)m_q(\mu_b).
\]
The odd moments vanish because $\mu_b$ is symmetric.
\end{theorem}

\begin{proof}
Reversing the order of the basis does not change normalized traces.  After
this reversal, the coefficient attached to the edge $\{k,k+1\}$ is
$\sigma_{k,N}b_k$.  Expanding the diagonal entries of the $q$th power gives
the exact identity
\begin{equation}\label{eq:path-trace-scale}
 \trN((X_N^\sigma)^q)
 =\frac1N\sum_{\gamma\in\mathcal C_q}
   \sum_{p\in I_N(\gamma)}
   \prod_{i\in\mathbb Z}
   (\sigma_{p+i,N}b_{p+i})^{\ell_i(\gamma)}.
\end{equation}
This is the specialization of the path monomial and trace decompositions in
\cite[Equations~(2.6) and (2.12)]{P09} to zero-diagonal matrices.

We first compute the expectation.  For $\gamma\in\mathcal C_q$, set
\[
 c_\gamma:=\prod_{i\in\mathbb Z}
       \E[b_1^{\ell_i(\gamma)}],
 \qquad
 F_\gamma((x_i)_{|i|\le q})
 :=\prod_{i\in\mathbb Z}x_i^{\ell_i(\gamma)}.
\]
The products are finite.  Independence of the $b_k$ in
\eqref{eq:path-trace-scale} gives
\begin{equation}\label{eq:expected-path-scale}
 \E\trN((X_N^\sigma)^q)
 =\sum_{\gamma\in\mathcal C_q}c_\gamma
   \frac1N\sum_{p\in I_N(\gamma)}
   F_\gamma((\sigma_{p+i,N})_{|i|\le q}).
\end{equation}
The omission of the $O_q(1)$ boundary bases has no effect on the limit;
this is the analogue of the boundary reduction in
\cite[Equation~(2.13)]{P09}.  The key point is that
\[
 F_\gamma(t,\ldots,t)=t^{\sum_i\ell_i(\gamma)}=t^q.
\]
Since $F_\gamma$ is bounded on the common range of the profiles,
Lemma~\ref{lem:local-block} yields
\[
 \frac1N\sum_{p\in I_N(\gamma)}
 F_\gamma((\sigma_{p+i,N})_{|i|\le q})
 \longrightarrow \E T^q.
\]
In the terminology of \cite{P09}, this is the counterpart of the decisive
local computation in Equation~(2.15).  There are no flat paths to remove
here, because the diagonal of the matrix is zero.

If $\sigma\equiv1$, the same path expansion gives
$m_q(\mu_b)=\sum_{\gamma\in\mathcal C_q}c_\gamma$.  Thus
\eqref{eq:expected-path-scale} implies
\[
 \E\trN((X_N^\sigma)^q)
 \longrightarrow \E T^q\,m_q(\mu_b)
 =m_q(\mathcal L(TX)).
\]
For odd $q$ the set $\mathcal C_q$ is empty, which also explains directly
why the odd moments vanish.

We next pass from expectation to almost-sure convergence.  Group in
\eqref{eq:path-trace-scale} all paths based at the same $p$, and denote the
centered contribution by $Z_{p,N}$.  The variables $Z_{p,N}$ are uniformly
bounded, and $Z_{p,N}$ depends only on
$b_{p-q},\ldots,b_{p+q}$.  Hence $Z_{p,N}$ and $Z_{s,N}$ are independent
whenever $|p-s|>2q$.  In the expansion of
$\E|\sum_p Z_{p,N}|^4$, a term vanishes if one of its four indices lies at
distance greater than $2q$ from the other three.  Only $O_q(N^2)$ index
quadruples remain, and consequently
\[
 \E\left|\trN((X_N^\sigma)^q)
       -\E\trN((X_N^\sigma)^q)\right|^4
 \le \frac{C_q}{N^2}.
\]
Borel--Cantelli gives almost-sure convergence of every moment.  This is the
finite-range version of the argument following Equation~(2.17) in
\cite{P09}.

Finally,
$\|X_N^\sigma\|_{\mathrm{op}}\le
2\sup_{k,N}|\sigma_{k,N}|\,\|b_1\|_\infty$ almost surely, and the same
bound holds for the limiting operator.  The measures therefore have a
common compact support.  Moment convergence determines the limiting measure
and proves both asserted modes of weak convergence.
\end{proof}

Theorem~\ref{thm:scale-mixture} is a transfer principle for
Theorem~\ref{T2}.  The latter contains the microscopic random information in
$\mu_b$, whereas the present theorem shows that a deterministic profile which
is locally constant at a uniformly chosen edge contributes only an
independent scalar factor $T$.  The marginal consequence of \eqref{H3}
records the macroscopic proportions of the profile values, while its joint
part says that microscopic interfaces occupy a negligible fraction of the
matrix.  The stronger vanishing-mesh condition \eqref{H3-infty} is convenient but not
intrinsic.  In terms of the proof of \cite{P09}, the first use of the
deformation assumptions is precisely the analogue of Equation~(2.15): a
fixed path sees only a bounded window, \eqref{H3} collapses that window to
one value, and its marginal then averages that value over the base point.

We finish the compact case by showing that every probability law can occur
as the macroscopic profile law $\mu_T$, even under the stronger uniform
vanishing-mesh condition.

\begin{proposition}\label{prop:quantile}
Let $T$ be a real random variable.  There exists a deterministic triangular
array $(\sigma_{k,N})_{1\le k\le N}$ such that
\[
 \max_{1\le k<N}|\sigma_{k+1,N}-\sigma_{k,N}|\longrightarrow0
\]
and $\nu_N=N^{-1}\sum_{k=1}^N\delta_{\sigma_{k,N}}\Rightarrow\mathcal L(T)$.
In particular, the array satisfies \eqref{H3}.  If $T\in L^2$, the same
array may be chosen so that
\[
 W_2(\nu_N,\mathcal L(T))\longrightarrow0,
 \qquad
 W_2(\overline\nu_N,\mathcal L(T))\longrightarrow0.
\]
\end{proposition}

\begin{proof}
Let $Q$ be the quantile function of $T$, and put
\[
 R_N:=N^{1/8},\qquad h_N:=N^{-1/8},\qquad
 K_N:=\left[\frac{2R_N}{h_N}\right]+2,
 \qquad m_N:=N-K_N.
\]
Here $[u]$ denotes the integer part of $u$, namely the largest integer not
exceeding $u$.

For the finitely many $N$ with $m_N\le0$, define the row arbitrarily.  For
all other $N$, let $c_N(x):=(-R_N)\vee(x\wedge R_N)$.

Take the ordered core points
\[
 x_{j,N}:=c_N\!\left(Q\!\left(\frac{j-\frac12}{m_N}\right)\right),
 \qquad 1\le j\le m_N.
\]
Between every two distinct consecutive core values, insert equally spaced
bridge points so that every gap is at most $h_N$.  Their number is at most
\[
 \left\lceil\frac{x_{m_N,N}-x_{1,N}}{h_N}\right\rceil
 \le\left\lceil\frac{2R_N}{h_N}\right\rceil<K_N.
\]
Fill unused slots with copies of $x_{1,N}$ and place them next to that
point.  The resulting row has exactly $N$ entries and satisfies
\[
 \max_{k<N}|\sigma_{k+1,N}-\sigma_{k,N}|
 \le h_N\longrightarrow0.
\]

Let $\alpha_N=m_N^{-1}\sum_{j=1}^{m_N}\delta_{x_{j,N}}$ and let
$\mu^{(R_N)}=\mathcal L(c_N(T))$.  Monotonicity of $Q$ and a direct
quantile coupling give
\[
 W_1(\alpha_N,\mu^{(R_N)})\le\frac{2R_N}{m_N},
 \qquad
 W_2^2(\alpha_N,\mu^{(R_N)})\le\frac{4R_N^2}{m_N}.
\]
Coupling the inserted points with arbitrary core points gives
\[
 W_1(\nu_N,\alpha_N)\le\frac{2R_NK_N}{N},
 \qquad
 W_2^2(\nu_N,\alpha_N)\le\frac{4R_N^2K_N}{N}.
\]
All four errors tend to zero for the chosen powers.  Since
$\mu^{(R_N)}\Rightarrow\mathcal L(T)$, weak convergence follows.  If
$T\in L^2$, then $W_2(\mu^{(R_N)},\mathcal L(T))\to0$ by dominated
truncation, and the displayed $W_2$ bounds prove the first stronger
conclusion.  Removing one point of a measure supported on $[-R_N,R_N]$
changes the squared $W_2$ distance by at most $4R_N^2/N=o(1)$, so the same
conclusion holds for $\overline\nu_N$.  Finally, the vanishing mesh and the
weak empirical convergence imply \eqref{H3} by
Remark~\ref{rem:strong-sigma}.
\end{proof}

\begin{remark}
Continuity of the distribution function alone does not make the raw quantile
mesh vanish: separated components of the support create fixed jumps, and
heavy tails can create growing extreme gaps.  The bridge construction deals
with both phenomena.  If only the averaged condition \eqref{H3}, rather than
\eqref{H3-infty}, is required, ordered empirical quantiles already
suffice.
\end{remark}

\subsection{The general tridiagonal model}

As in the compact case above, the diagonal does not affect the limiting distribution.

\begin{lemma}\label{deform}
Let $X_N^\sigma$ satisfy (H2), and let $\widetilde X_N^\sigma$ be
obtained by setting its diagonal to zero.  For every $z\in\C^+$,
\[
 \E|S_{X_N^\sigma}(z)-S_{\widetilde X_N^\sigma}(z)|\longrightarrow0.
\]
Consequently the two models have the same almost-sure weak limit whenever
one of the expected Stieltjes transforms has a limit.
\end{lemma}

\begin{proof}
The resolvent identity, Cauchy--Schwarz, and Lemma~\ref{HW} give
\[
 |S_A(z)-S_B(z)|
 \le\frac{\|A-B\|_{\mathrm{HS}}}{\sqrt N(\Im z)^2}.
\]
Since
\[
 \E\|X_N^\sigma-\widetilde X_N^\sigma\|_{\mathrm{HS}}^2
 =\sum_{i=1}^N\E|a_{i,N}|^2=o(N),
\]
the first assertion follows.  The martingale concentration estimate applies
to both models, and hence transfers the deterministic expected limit to an
almost-sure limit on a countable Stieltjes set.
\end{proof}

Hence we may (and will) study $X_N^\sigma$ directly in the case where $(a_{i,N})_{i,N}$ is the zero family.
\medskip

We also note how the parameter $\alpha$ appears in the limiting distribution
as a deformation in Theorem~\ref{T1}.  The case $\sigma\equiv1$ contains the
microscopic information needed to describe the limiting spectral distribution
of $X_N^\sigma$; the remaining information is the local empirical behavior of
$\sigma$.
\medskip

We now remove the boundedness assumptions from
Theorem~\ref{thm:scale-mixture}.
\medskip

\begin{theorem}\label{T3}
Let $X_N^\sigma$ be the model \eqref{model}.  Under
Assumptions~\ref{assumptionmatrix} and condition~\eqref{H3},
\[
 \mu_{X_N^\sigma}\Rightarrow\mathcal L(TX)
 \qquad\text{almost surely},
\]
where $T\sim\mu_T$, $X\sim\mu_b$, and $T,X$ are independent.  No
moment assumption on $T$ is required.
If, in addition, $T\in L^2$ and \eqref{H3-2} holds, then
\[
 W_2\bigl(\mu_{X_N^\sigma},\mathcal L(TX)\bigr)
 \longrightarrow0
 \qquad\text{in probability}.
\]
\end{theorem}

\begin{proof}
By Lemma~\ref{deform}, it is enough to treat the zero-diagonal model; all
matrices below in this proof have zero diagonal.
For $R>0$, let $c_R(x):=(-R)\vee(x\wedge R)$ and set
\[
 b_i^{(C)}:=c_C(b_i),\qquad
 \sigma_{i,N}^{(M)}:=c_M(\sigma_{i,N}).
\]
The clipped deformation still satisfies \eqref{H3}, with limit
$c_M(T)$, because $c_M$ is continuous and $1$-Lipschitz.  Hence
Theorem~\ref{thm:scale-mixture} gives, for fixed $C,M$,
\[
 \mu_{X_{N,C}^{\sigma^{(M)}}}
 \Rightarrow\mathcal L(c_M(T)X_C)
 \qquad\text{almost surely},
\]
where $X_C$ has the simple-model limit associated with $b_1^{(C)}$.

Changing one off-diagonal edge changes the matrix by rank at most two.
Therefore \eqref{eq:rank} gives
\[
 \sup_x\big|F_{X_N^\sigma}(x)-F_{X_{N,C}^{\sigma^{(M)}}}(x)\big|
 \le\frac2N\#\{i<N:|b_i|>C\}
 +\frac2N\#\{i<N:|\sigma_{i,N}|>M\}.
\]
For every continuity point $M$ of the tail of $\mu_T$, the strong law and
the first marginal of \eqref{H3} yield
\[
 \limsup_{N\to\infty}\sup_x|F_{X_N^\sigma}(x)
       -F_{X_{N,C}^{\sigma^{(M)}}}(x)|
 \le2\mathbb P(|b_1|>C)+2\mu_T(|t|>M)
 \qquad\text{almost surely}.
\]
Choose countable sequences $C_j,M_j\to\infty$ such that every $M_j$ is a
continuity point of $M\mapsto\mu_T(\{|t|>M\})$.  On the common
probability-one event furnished by the preceding estimates and by the
almost-sure bounded-model convergence for the countably many pairs
$(C_j,M_j)$, the right-hand side tends to zero along these sequences.

The same rank comparison for the simple models gives
$\mu_{X_C}\Rightarrow\mu_b$.  It follows that
$\mathcal L(c_M(T)X_C)\Rightarrow\mathcal L(TX)$, first as $C\to\infty$
and then as $M\to\infty$.  A diagonal argument in the bounded-Lipschitz
metric completes the proof of weak convergence.

For the $W_2$ assertion we return to the original matrix, including its
diagonal, and assume $T\in L^2$ and \eqref{H3-2}.  Wasserstein convergence of
$\overline\nu_N$ implies
\begin{equation}\label{eq:profile-second-moment}
 \frac1N\sum_{k=1}^{N-1}\sigma_{k,N}^2\longrightarrow\E T^2,
 \qquad
 \max_{k<N}\frac{\sigma_{k,N}^2}{N}\longrightarrow0.
\end{equation}
For the second assertion, split the maximum at a fixed level $M$ and use
uniform integrability of the squared profile values on the set
$\{|\sigma_{k,N}|>M\}$.

Put $w_{k,N}:=\sigma_{k,N}^2/N$ and $Y_k:=b_k^2$.  We claim that
\begin{equation}\label{eq:weighted-weak-law}
 \sum_{k=1}^{N-1}w_{k,N}(Y_k-\E Y_1)
 \longrightarrow0
 \qquad\text{in probability}.
\end{equation}
Indeed, for $L>0$ replace $Y_k$ by $Y_k\wedge L$.  Independence and
\eqref{eq:profile-second-moment} give
\[
 \operatorname{Var}\!\left(
  \sum_{k<N}w_{k,N}\bigl(Y_k\wedge L-\E(Y_1\wedge L)\bigr)
 \right)
 \le L^2\Bigl(\max_{k<N}w_{k,N}\Bigr)\sum_{k<N}w_{k,N}
 \longrightarrow0.
\]
The expectation of the absolute contribution of the discarded tails is at
most
\[
 2\Bigl(\sum_{k<N}w_{k,N}\Bigr)
 \E\bigl[Y_1\ind_{\{Y_1>L\}}\bigr],
\]
which tends to zero uniformly in $N$ as $L\to\infty$.  This proves
\eqref{eq:weighted-weak-law}.

The exact second-moment identity is
\begin{equation}\label{eq:matrix-second-moment}
 \int x^2\,d\mu_{X_N^\sigma}(x)
 =\frac1N\sum_{i=1}^Na_{i,N}^2
  +\frac2N\sum_{k=1}^{N-1}\sigma_{k,N}^2b_k^2.
\end{equation}
The first term tends to zero in $L^1$ by (H2), while
\eqref{eq:profile-second-moment}--\eqref{eq:weighted-weak-law} show that
the second tends in probability to
$2\E T^2\E b_1^2$.  On the other hand,
Proposition~\ref{prop:simple-W2} gives
$m_2(\mu_b)=2\E b_1^2$, and therefore
\[
 \E(TX)^2=2\E T^2\E b_1^2.
\]
To pass from these two modes of convergence to the asserted random
$W_2$ convergence, take an arbitrary subsequence.  Convergence of the
second moments in probability provides a further subsequence on which they
converge almost surely.  On the probability-one event where the weak
convergence already holds, the deterministic characterization of $W_2$
convergence then gives $W_2\to0$ along that further subsequence.  This
subsequence criterion proves convergence in probability for the full
sequence.
\end{proof}

\begin{remark}[Why $T\in L^2$ is not sufficient]\label{rem:W2-obstruction}
The additional condition \eqref{H3-2} cannot be omitted.  Take
$a_{i,N}=0$, $b_k\equiv1$, and
\[
 \sigma_{1,N}=\sqrt N,
 \qquad \sigma_{k,N}=1\quad(2\le k\le N).
\]
Then \eqref{H3} holds with $T\equiv1$, and changing one edge does not alter
the weak limiting ESD.  Nevertheless,
\[
 \frac1N\operatorname{Tr}((X_N^\sigma)^2)
 =\frac2N\sum_{k=1}^{N-1}\sigma_{k,N}^2\longrightarrow4,
\]
whereas $m_2(\mu_b)=2$.  Thus weak convergence does not improve to
$W_2$ convergence.  The mode in the second part of Theorem~\ref{T3} is in
probability because (H2) controls the squared Hilbert--Schmidt norm of the
diagonal in $L^1$, but not necessarily almost surely, and because the
weighted law \eqref{eq:weighted-weak-law} is available in probability under
the stated triangular-array hypotheses.
\end{remark}

The coupling lemma used after Theorem~\ref{T2} gives the corresponding
extension to non-identically distributed off-diagonal entries.

\begin{corollary}\label{cor:deformed-noniid}
Let $(b_k)_{k\ge1}$ be independent real random variables such that
\[
 W_2(\mathcal L(b_k),\mathcal L(b))\longrightarrow0
\]
for some $b\in L^2$.  Suppose that (H2), \eqref{H3}, $T\in L^2$, and
\eqref{H3-2} hold.  Then
\[
 W_2\bigl(\mu_{X_N^\sigma},\mathcal L(TX)\bigr)
 \longrightarrow0
 \qquad\text{in probability},
\]
where $X$ has the simple-model law $\mu_b$ associated with $b$, and $T$
and $X$ are independent.
\end{corollary}

\begin{proof}
Take the product coupling furnished by Lemma~\ref{e:W_2err}, independently
of the diagonal family, and write
$b_k=\beta_k+\varepsilon_k$, where the $\beta_k$ are i.i.d. with law
$\mathcal L(b)$ and
\[
 d_k^2:=\E|\varepsilon_k|^2
 =W_2^2(\mathcal L(b_k),\mathcal L(b))\longrightarrow0.
\]
Let $Y_N^\sigma$ be the matrix obtained from $X_N^\sigma$ by replacing
$b_k$ with $\beta_k$ and leaving the diagonal unchanged.  The
Hoffman--Wielandt inequality gives
\[
 \E W_2^2(\mu_{X_N^\sigma},\mu_{Y_N^\sigma})
 \le\frac2N\sum_{k=1}^{N-1}\sigma_{k,N}^2d_k^2.
\]
With $w_{k,N}=\sigma_{k,N}^2/N$, condition \eqref{H3-2} gives
$\sum_kw_{k,N}=O(1)$ and $\max_kw_{k,N}\to0$.  For a fixed $K$,
\[
 \sum_{k<N}w_{k,N}d_k^2
 \le \Bigl(\max_{k<N}w_{k,N}\Bigr)\sum_{k<K}d_k^2
 +\Bigl(\sum_{k<N}w_{k,N}\Bigr)\sup_{k\ge K}d_k^2.
\]
First let $N\to\infty$ and then $K\to\infty$.  The right-hand side tends
to zero, so the two ESDs are asymptotically equal in $W_2$ in probability.
The second part of Theorem~\ref{T3} applies to $Y_N^\sigma$, and the
triangle inequality completes the proof.
\end{proof}

\subsection{Examples}\label{s:5}

We now present a few simple yet somewhat exotic examples that appear to be less commonly discussed in the literature.

\paragraph{Constant coefficients and the Ullman laws}

\begin{proposition}[Power profile and Ullman law]\label{prop:ullman}
Let $\alpha\ge0$ and consider the zero-diagonal model with $b_k\equiv1$
and
\[
 \sigma_{k,N}=\left(\frac{k}{N}\right)^\alpha,
 \qquad 1\le k\le N.
\]
Equivalently, this is the matrix obtained by dividing by $N^\alpha$ a
tridiagonal matrix whose $k$th off-diagonal coefficient is $k^\alpha$.
Then
\[
 \mu_{X_N^\alpha}\Rightarrow\mathcal L(U^\alpha A)
 \qquad\text{almost surely},
\]
where $U$ is uniform on $[0,1]$, $A$ has the arcsine law $\gamma_2$ on
$[-2,2]$, and $U,A$ are independent.  This law is usually called the
Ullman distribution with parameter $\alpha$.
\end{proposition}

\begin{proof}
We first identify the limit of the simple model $b_k\equiv1$.

\smallskip
\noindent\emph{Toeplitz computation.}
The zero-diagonal tridiagonal Toeplitz matrix with off-diagonal entries
equal to $1$ has eigenvalues
\[
 2\cos\left(\frac{j\pi}{N+1}\right),
 \qquad 1\le j\le N.
\]
Consequently its ESD converges to the arcsine law $\gamma_2$ on $[-2,2]$.

\smallskip
\noindent\emph{Stieltjes-transform computation.}
Alternatively, the cavity transform satisfies
\[
 s(z)=\frac{1}{z-s(z)},
 \qquad
 s(z)=\frac{z-\sqrt{z^2-4}}{2},
\]
where the branch is chosen so that $s(z)\sim z^{-1}$ at infinity.
Theorem~\ref{T2} therefore gives
\[
 S_{\gamma_2}(z)=\frac{1}{z-2s(z)}=\frac{1}{\sqrt{z^2-4}}.
\]

Finally, the empirical measure of
$\sigma_{k,N}=(k/N)^\alpha$ converges to $\mathcal L(U^\alpha)$ by
Riemann sums.  For $\alpha=0$ the profile is constant.  For $\alpha>0$,
uniform continuity of $x\mapsto x^\alpha$ on $[0,1]$ gives
\[
 \max_{k<N}|\sigma_{k+1,N}-\sigma_{k,N}|\longrightarrow0.
\]
The conclusion follows from Theorem~\ref{T3}.
\end{proof}

\begin{figure}[ht!]
\centering
\includegraphics[width=4.2cm]{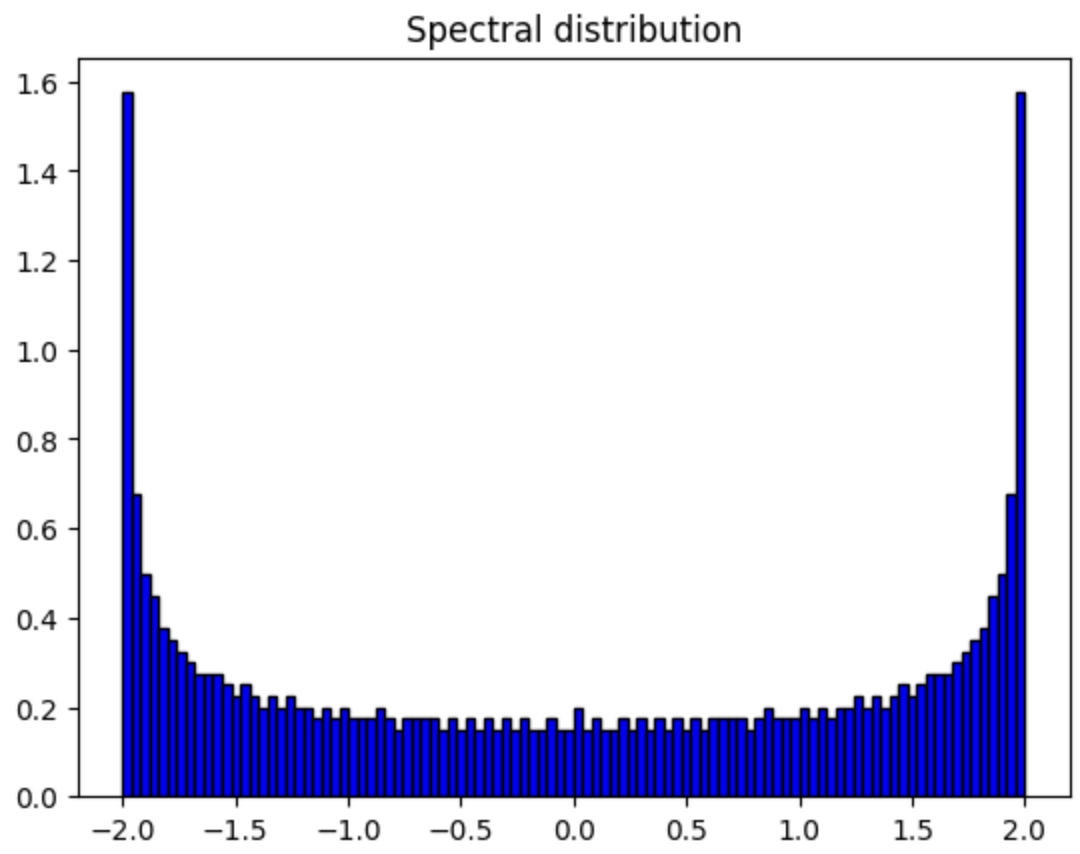}
\includegraphics[width=4.2cm]{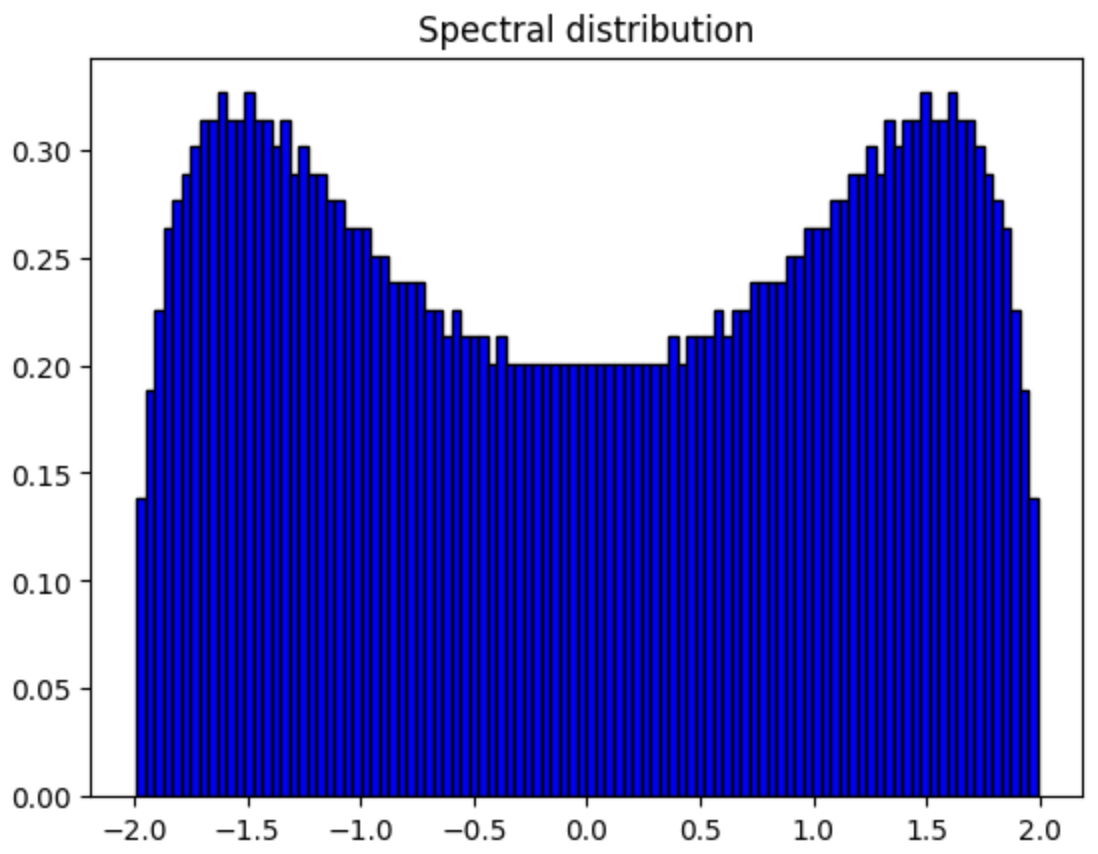}
\includegraphics[width=4.2cm]{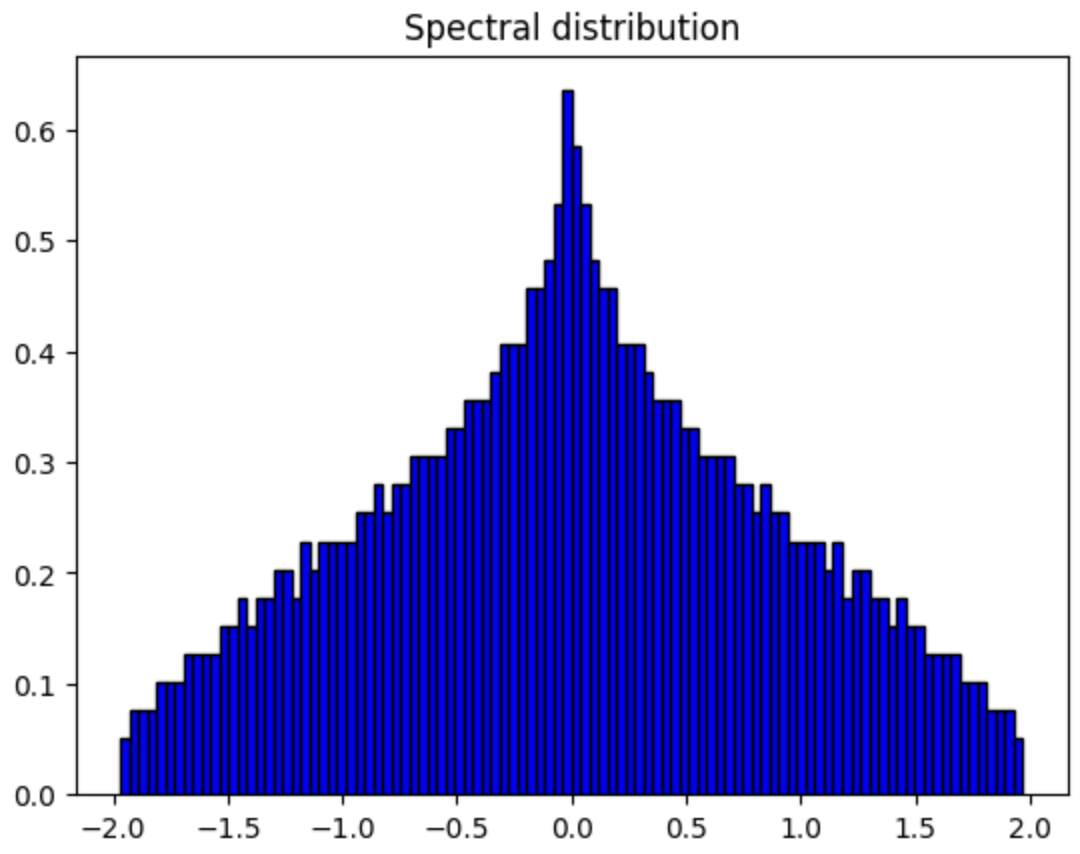}
\caption{Normalized eigenvalue histograms for the power profiles with
$\alpha=0$, $\alpha=0.2$, and $\alpha=0.8$ (from left to right), for
$N=2000$.  The panels visualize the scale mixtures $U^\alpha A$ identified
in the proposition.}
\label{fig:ullman}
\end{figure}

At $\alpha=0$, Figure~\ref{fig:ullman} recovers the arcsine density and its
singular behavior at the two edges.  For $\alpha>0$, the factor
$U^\alpha$ takes all values in $[0,1]$ and superposes copies of the arcsine
law at different scales.  As $\alpha$ increases, $U^\alpha$ puts more mass
near zero, so the edge peaks are progressively replaced by a central
cusp.  The common support remains $[-2,2]$, and symmetry is preserved.
This gives a concrete picture of the scale-mixture mechanism in
Theorem~\ref{T3}, rather than only a change of variables in the moment
formula.

\paragraph{Bernoulli entries}

\begin{proposition}[Limiting distribution for Bernoulli entries]
Let $(b_k)_{k\ge1}$ be i.i.d.\ Bernoulli variables with
\[
 \p(b_k=1)=p,\qquad \p(b_k=0)=q:=1-p,
 \qquad 0\le p<1.
\]
For $k\ge1$, let $A_k$ be the $k\times k$ zero-diagonal tridiagonal
Toeplitz matrix with off-diagonal entries equal to $1$, and let
\[
 \mu_k:=\frac1k\sum_{j=1}^k
 \delta_{2\cos(j\pi/(k+1))}
\]
be its normalized spectral distribution.  Then
\[
 \mu_{X_N}\Rightarrow\mu_{B_p}\qquad\text{almost surely},
 \qquad
 \mu_{B_p}=q^2\sum_{k=1}^{\infty}k\,p^{k-1}\mu_k.
\]
The coefficients in this mixture sum to one.
\end{proposition}

\begin{proof}
Every zero off-diagonal entry disconnects the path, so $X_N$ is an
orthogonal direct sum of matrices of the form $A_k$, apart from the two
boundary components.  An interior component has $k$ vertices exactly when
the edge pattern is
\[
 0,\underbrace{1,\ldots,1}_{k-1\text{ entries}},0.
\]
By the ergodic theorem, the number of such components divided by $N$
converges almost surely to $q^2p^{k-1}$.  Each component contributes $k$
eigenvalues.  The two boundary components have asymptotically negligible
size, and hence
\[
 \mu_{X_N}\Rightarrow q^2\sum_{k\ge1}k\,p^{k-1}\mu_k.
\]
Finally,
\[
 q^2\sum_{k\ge1}k\,p^{k-1}=\frac{q^2}{(1-p)^2}=1.
\]
\end{proof}

When $p=1$, all edges are present and the model is the
constant-coefficient Toeplitz matrix treated in
Proposition~\ref{prop:ullman}; its limit is the arcsine law on $[-2,2]$.

\begin{remark}
If $P_k(z):=\det(zI_k-A_k)$, then
$P_k'(z)/P_k(z)=kS_{\mu_k}(z)$, and therefore
\[
 S_{\mu_{B_p}}(z)
 =q^2\sum_{k=1}^{\infty}p^{k-1}\frac{P_k'(z)}{P_k(z)}.
\]
This agrees with Theorem~\ref{T2}.  Indeed, if
$B\sim\operatorname{Bernoulli}(p)$ and $R'$ is an independent copy of
$R$, then
\[
 R\stackrel d=\frac{1}{z-BR'},
\]
and, for independent copies $R_1,R_2$,
\[
 S_{\mu_{B_p}}(z)
 =\frac{q^2}{z}
  +2pq\,\E\left[\frac{1}{z-R}\right]
  +p^2\,\E\left[\frac{1}{z-R_1-R_2}\right].
\]
\end{remark}

\begin{figure}[ht!]
\centering
\includegraphics[height=5cm]{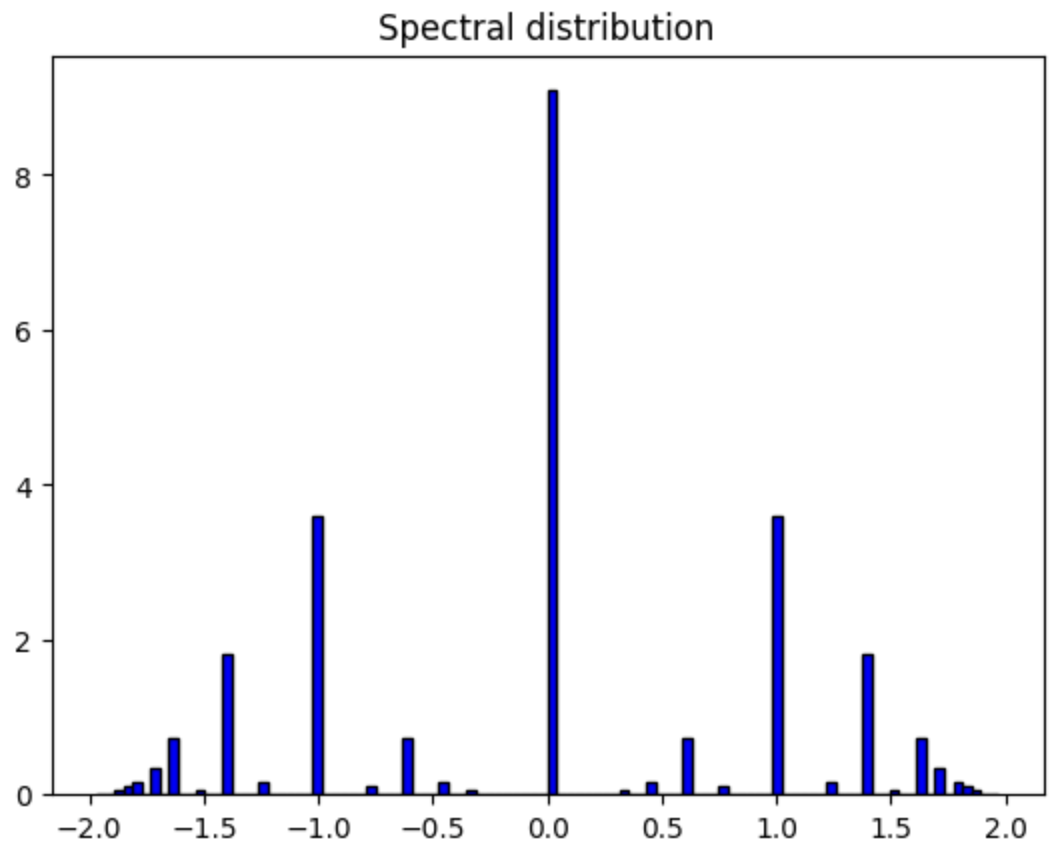}
\includegraphics[height=5cm]{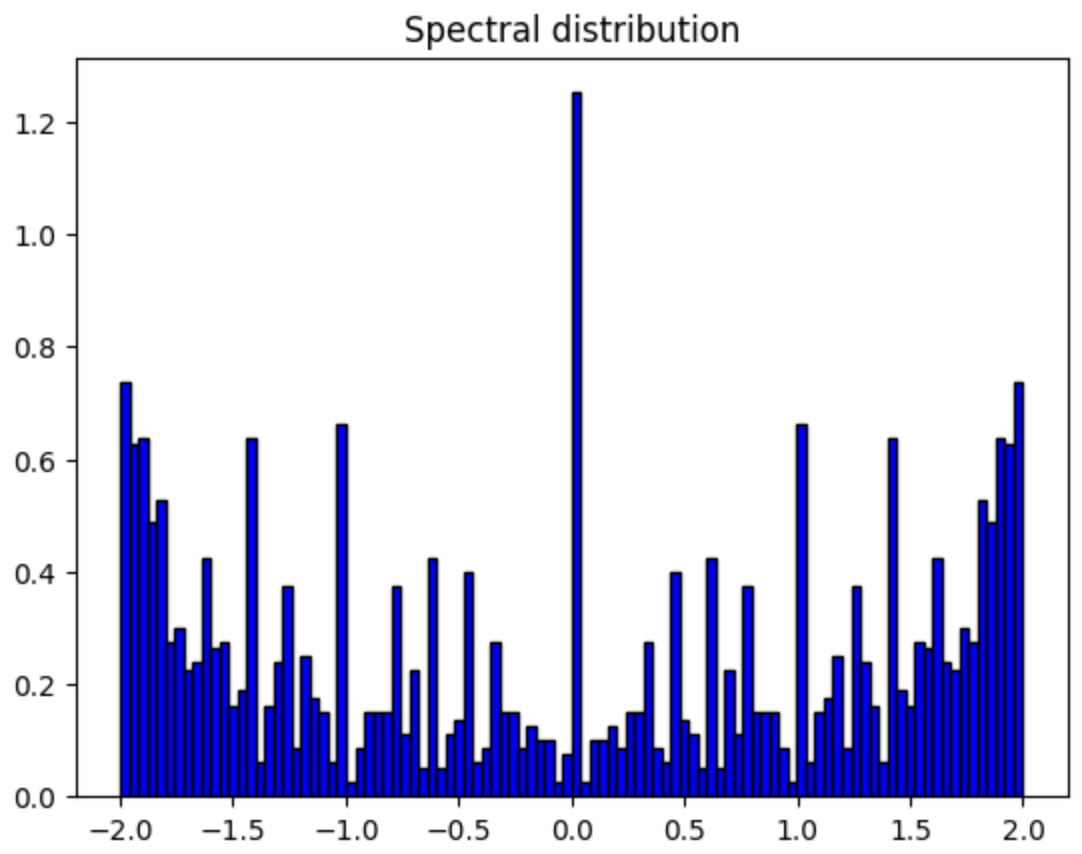}
\caption{Normalized eigenvalue histograms for Bernoulli edges with $p=0.5$
(left) and $p=0.9$ (right).  The spikes are the eigenvalues of the finite
path components appearing in the discrete mixture $\mu_{B_p}$.}
\label{fig:bernoulli}
\end{figure}

The contrast in Figure~\ref{fig:bernoulli} is explained directly by the
block decomposition in the proof.  When $p=0.5$, zeros frequently cut the
matrix into short paths, and a small collection of finite-path eigenvalues
therefore carries substantial mass.  When $p=0.9$, longer runs occur more
often; their eigenvalues form a denser subset of $[-2,2]$, and the envelope
begins to resemble the arcsine law obtained at $p=1$.  The law nevertheless
remains purely atomic for every $p<1$.  In particular, odd path components
produce the conspicuous atom at zero.

\section{Joint limits and the role of the shift}\label{s:4}

\subsection{A deterministic profile and a tridiagonal matrix}

It is useful to complement the scale-mixture theorem by a joint-moment
statement.  Let $\mathcal R_N$ be the permutation matrix which reverses the
standard basis, and set
\[
 \widehat X_N:=\mathcal R_NX_N\mathcal R_N^T,
 \qquad
 D_N:=\operatorname{diag}(\sigma_{1,N},\ldots,\sigma_{N,N}).
\]
Thus the entry $b_k$ of $\widehat X_N$ lies between vertices $k$ and $k+1$.
This reversal preserves the marginal normalized traces and ESD of $X_N$ and
places its coefficients in the coordinate order used by $D_N$.

\begin{theorem}[Joint convergence of the profile and the simple model]
\label{thm:joint-sigma-X}
Assume the hypotheses of Theorem~\ref{thm:scale-mixture}; in particular,
the variables $b_k$ and the deformation array are uniformly bounded.  Let
$T\sim\mu_T$ and $X\sim\mu_b$ be independent.  Then, for every
noncommutative polynomial $P$ in two indeterminates,
\[
 \trN P(D_N,\widehat X_N)
 \longrightarrow \E[P(T,X)]
\]
in expectation and almost surely, where on the right the polynomial is
evaluated at the commuting scalar variables $T$ and $X$.  In particular,
for all $r,q\ge0$,
\[
 \trN(D_N^r\widehat X_N^q)
 \longrightarrow m_r(\mu_T)m_q(\mu_b)
 \qquad\text{almost surely}.
\]
Thus the limiting joint distribution is commutative and $T,X$ are
classically independent.
\end{theorem}

\begin{proof}
By linearity, it is enough to consider a word $W$ containing $r$
occurrences of $D_N$ and $q$ occurrences of $\widehat X_N$.  Using
cyclicity of the normalized trace, and allowing zero exponents, we may write
the word in the form
\[
 W=D_N^{a_0}\widehat X_ND_N^{a_1}\widehat X_N
       \cdots\widehat X_ND_N^{a_q},
 \qquad a_0+\cdots+a_q=r.
\]
The case $q=0$ is simply the marginal convergence following from
\eqref{H3}, so assume $q\ge1$.

For $\gamma=(j_0,\ldots,j_q)\in\mathcal C_q$ define
\[
 H_{\gamma,a}((x_i)_{|i|\le q})
 :=\prod_{s=0}^q x_{j_s}^{a_s}.
\]
Expanding the trace entry by entry, exactly as in
\eqref{eq:path-trace-scale}, gives
\begin{align}\label{eq:path-trace-profile-joint}
 \trN W
 =\frac1N\sum_{\gamma\in\mathcal C_q}
  \sum_{p\in I_N(\gamma)}
  H_{\gamma,a}((\sigma_{p+i,N})_{|i|\le q})
  \prod_{i\in\mathbb Z}b_{p+i}^{\ell_i(\gamma)}.
\end{align}
Thus the order of the $D_N$ factors is retained in the function
$H_{\gamma,a}$; it has not yet disappeared.

We first take expectations.  Independence of the edge variables gives the
same coefficient $c_\gamma$ as in the proof of
Theorem~\ref{thm:scale-mixture}.  Since
\[
 H_{\gamma,a}(t,\ldots,t)=t^{a_0+\cdots+a_q}=t^r,
\]
Lemma~\ref{lem:local-block}, together with the $O_q(1)$ boundary estimate,
implies
\[
 \frac1N\sum_{p\in I_N(\gamma)}
 H_{\gamma,a}((\sigma_{p+i,N})_{|i|\le q})
 \longrightarrow\E T^r.
\]
Consequently,
\[
 \E\trN W
 \longrightarrow \E T^r\sum_{\gamma\in\mathcal C_q}c_\gamma
 =\E T^r\,m_q(\mu_b).
\]
This calculation also explains why all words containing the same numbers
$r$ and $q$ of the two letters have the same limit: condition \eqref{H3}
collapses every finite profile window to a single value.

For the almost-sure assertion, group the centered terms in
\eqref{eq:path-trace-profile-joint} according to the base point $p$.
Each group is uniformly bounded and depends only on
$b_{p-q},\ldots,b_{p+q}$.  The fourth-moment argument used in
Theorem~\ref{thm:scale-mixture} therefore gives
\[
 \E\left|\trN W-\E\trN W\right|^4\le \frac{C_W}{N^2}.
\]
Borel--Cantelli proves almost-sure convergence for each word, and then for
every polynomial.  The common compact support identifies the limiting
functional with the joint law of the commuting variables $(T,X)$.  Since
its mixed moments factor as $\E T^r\E X^q$, these variables are classically
independent.
\end{proof}

Theorem~\ref{thm:joint-sigma-X} clarifies the role of
Theorem~\ref{thm:scale-mixture}.  The simple model contributes the
microscopic random law $\mu_b$, while the deterministic profile contributes
the scalar $T$.  Under the signed local condition \eqref{H3}, the two pieces
become commuting and classically independent at the level of normalized
traces.  This is different from Voiculescu freeness.  For the spectral
scale-mixture result alone, signs can instead be removed as in
Remark~\ref{rem:absolute-sigma}.

\subsection{Several tridiagonal matrices and a common-shift algebra}
\label{s:6}

The joint limits of several simple models have a stationary-operator
realization.  Let $S$ be the bilateral shift on $\ell^2(\mathbb Z)$,
$Se_j=e_{j+1}$.  For $1\le u\le r$, let
$(b_j^{(u)})_{j\in\mathbb Z}$ be a bounded i.i.d.\ sequence, assume that
all these sequences are mutually independent, and define
\[
 B_ue_j=b_j^{(u)}e_j,
 \qquad
 J_u=SB_u+B_uS^*.
\]
On the algebra of covariant random operators generated by $S$ and the
coefficient fields, use the canonical tracial state
\[
\Phi(A):=\E\langle e_0,Ae_0\rangle.
\]
Stationarity makes $\Phi$ tracial.  Indeed, for finite-propagation covariant
operators, expand the matrix coefficient of a product at $e_0$, shift each
summand to the origin, and reindex the finite sum; this gives
$\Phi(AC)=\Phi(CA)$.  The identity extends by continuity.

For a color word $c=(c_1,\ldots,c_q)\in\{1,\ldots,r\}^q$,
$\gamma\in\mathcal C_q$, an edge $\{i,i+1\}$, and a color $u$, put
\[
 \ell_{i,u}(\gamma,c):=
 \#\{1\le s\le q:c_s=u,\
 \{j_{s-1},j_s\}=\{i,i+1\}\}.
\]

\begin{theorem}[Colored-path joint moments]\label{joint}
Let $X_{1,N},\ldots,X_{r,N}$ be independent zero-diagonal simple
tridiagonal models whose respective off-diagonal laws are those of
$b_0^{(1)},\ldots,b_0^{(r)}$, and assume these variables are bounded.
For every color word $c=(c_1,\ldots,c_q)$,
\begin{align*}
 &\trN(X_{c_1,N}\cdots X_{c_q,N})
 \longrightarrow \Phi(J_{c_1}\cdots J_{c_q})\\
 &\qquad=
 \sum_{\gamma\in\mathcal C_q}
 \prod_{i\in\mathbb Z}\prod_{u=1}^r
 \E\!\left[(b_0^{(u)})^{\ell_{i,u}(\gamma,c)}\right]
\end{align*}
in expectation and almost surely.  Only finitely many factors in the
product differ from $1$.
\end{theorem}

\begin{proof}
Conjugate all the finite matrices by the same reversal permutation and
retain the notation $X_{u,N}$.  The normalized trace of their product is
unchanged, and $b_k^{(u)}$ now lies on the edge $\{k,k+1\}$ for every
color $u$.  For $\gamma=(j_0,\ldots,j_q)\in\mathcal C_q$ and a step $s$,
the edge
traversed at that step has relative index
$\min(j_{s-1},j_s)$.  Direct expansion of the normalized trace gives
\begin{align}\label{eq:colored-path-finite}
 &\trN(X_{c_1,N}\cdots X_{c_q,N})\\
 &\quad=\frac1N\sum_{\gamma\in\mathcal C_q}
 \sum_{p\in I_N(\gamma)}
 \prod_{s=1}^q
 b_{p+\min(j_{s-1},j_s)}^{(c_s)}.\notag
\end{align}
There are only finitely many path shapes.  For a fixed $\gamma$ and an
interior base point $p$, group the factors in the last product according to
the pair consisting of their spatial edge and their color.  Independence
over these pairs gives
\[
 \E\prod_{s=1}^q
 b_{p+\min(j_{s-1},j_s)}^{(c_s)}
 =\prod_{i\in\mathbb Z}\prod_{u=1}^r
 \E[(b_0^{(u)})^{\ell_{i,u}(\gamma,c)}].
\]
The right-hand side is independent of $p$.  Since
$|I_N(\gamma)|=N+O_q(1)$, taking expectation in
\eqref{eq:colored-path-finite} proves convergence to the displayed path
sum in the theorem.

We next identify this sum with the stationary-operator expression.  Set
\[
 M_c:=\langle e_0,J_{c_1}\cdots J_{c_q}e_0\rangle.
\]
Expand $M_c$ by choosing at each factor the forward or backward part of
$J_{c_s}$.  A nonzero diagonal coefficient is obtained exactly when these
choices form a path in $\mathcal C_q$, and its expected coefficient is the
product above.  Hence the rooted infinite operator gives precisely the
same sum.

Finally, after centering the contribution based at $p$ in
\eqref{eq:colored-path-finite}, one obtains a uniformly bounded variable
depending only on the colored edge variables with indices between $p-q$
and $p+q$.  Contributions based more than $2q$ apart are independent.
The same count of nonvanishing index quadruples as before yields
\[
 \E\left|\trN(X_{c_1,N}\cdots X_{c_q,N})
 -\E\trN(X_{c_1,N}\cdots X_{c_q,N})\right|^4
 \le\frac{C_{q,c}}{N^2}.
\]
Borel--Cantelli proves the almost-sure limit.  This is the bounded,
zero-diagonal form of \cite[Theorem~4]{P09}; the operator realization is
the stationary counterpart of \cite[Equations~(2.19) and
(2.23)--(2.24)]{P09}.
\end{proof}

As in the discussion following Theorem~4 of \cite{P09}, a few low-order
moments make the dependence particularly transparent.  For two colors
$a,b$, write
$m_k^{(a)}:=\E[(b_0^{(a)})^k]$ and similarly for $b$.  The path formula
gives
\begin{align*}
 \Phi(J_a^2)&=2m_2^{(a)},
 &\Phi(J_aJ_b)&=2m_1^{(a)}m_1^{(b)},\\
 \Phi(J_a^4)&=2m_4^{(a)}+4(m_2^{(a)})^2,
 &\Phi(J_aJ_bJ_aJ_b)
 &=2m_2^{(a)}m_2^{(b)}
   +4(m_1^{(a)})^2(m_1^{(b)})^2,\\
 \Phi(J_a^2J_b^2)
 &=4m_2^{(a)}m_2^{(b)}
   +2(m_1^{(a)})^2(m_1^{(b)})^2,
 &\Phi(J_a^3J_b)
 &=2m_3^{(a)}m_1^{(b)}
   +4m_2^{(a)}m_1^{(a)}m_1^{(b)}.
\end{align*}

The formula explains why ordinary spectral measures do not determine
joint limits.  For one matrix, every edge in a closed path is crossed an
even number of times, so the limiting spectral law is insensitive to signs
of the coefficients.  In a colored path, however, a fixed color may cross
an edge an odd number of times.  Odd coefficient moments can therefore
enter mixed moments even though they are invisible in the individual
spectral laws.

At the operator level,
\[
 J_1+J_2=S(B_1+B_2)+(B_1+B_2)S^*.
\]
Thus the sum of two independent simple models is again a simple model,
with edge law equal to the classical convolution of the two coefficient
laws.  This operation does not descend to a convolution of the limiting
spectral measures alone.  For example, the constant edge laws $\delta_1$
and $\delta_{-1}$ give the same arcsine spectral law, but adding a constant
edge $1$ gives respectively a model with constant edge $2$ and the zero
matrix.  The common-shift data retain the information lost by the marginal
spectral laws.

\subsection{The profile--shift pair and a general convergence theorem}

The deterministic assumptions can be expressed in terms of $D_N$ and the
truncated shift
\[
 S_Ne_k=e_{k+1}\quad(1\le k<N),
 \qquad S_Ne_N=0.
\]
The marginal convergence of $D_N$ is already a consequence of
\eqref{H3}.  More importantly, the same condition is equivalently encoded
by the mixed trace limits
\begin{equation}\label{eq:shift-mixed}
 \frac1N\operatorname{Tr}
 \bigl(f(D_N)S_N^*g(D_N)S_N\bigr)
 \longrightarrow\int f(t)g(t)\,d\mu_T(t),
 \qquad f,g\in C_b(\mathbb R).
\end{equation}
Indeed, the trace on the left is
\[
 \frac1N\sum_{k=1}^{N-1}
 f(\sigma_{k,N})g(\sigma_{k+1,N}).
\]
The endpoint normalization is immaterial.  These mixed products determine
the limiting pair measure.  Equivalently,
if $I_N$ is uniform among the indices whose distance from the boundary
exceeds a fixed $r$, then Lemma~\ref{lem:local-block} gives
\[
 (\sigma_{I_N-r,N},\ldots,\sigma_{I_N+r,N})
 \Rightarrow(T,\ldots,T).
\]
Thus \eqref{H3} is a uniformly rooted local-limit condition.  This point of
view is closely related to the local weak convergence of rooted random
networks, although the present one-dimensional deterministic setting is
considerably simpler; see, for example, \cite{AL07,BL10}.

There is also a useful commutator formulation.  For every bounded
Lipschitz function $f$,
\[
 \|[f(D_N),S_N]\|_{2,N}^2
 =\frac1N\sum_{k=1}^{N-1}
 |f(\sigma_{k+1,N})-f(\sigma_{k,N})|^2,
 \qquad
 \|A\|_{2,N}^2:=\trN(A^*A).
\]
Consequently, \eqref{H3} implies asymptotic shift invariance of $f(D_N)$.
If the profile is uniformly bounded, $f$ may be chosen to agree with the
identity on its common range, and hence
\begin{equation}\label{eq:commutator-two}
 \|[D_N,S_N]\|_{2,N}\longrightarrow0.
\end{equation}
By contrast, the stronger condition \eqref{H3-infty} gives the operator-norm
identity
\[
 \|[D_N,S_N]\|_{\mathrm{op}}
 =\max_{1\le k<N}|\sigma_{k+1,N}-\sigma_{k,N}|
 \longrightarrow0.
\]
Thus, for uniformly bounded profiles and with the marginal convergence
understood, the difference between the local-coherence parts of
\eqref{H3} and \eqref{H3-infty} is the difference between normalized
Hilbert--Schmidt and operator-norm control of the commutator.

In the bounded case, this already determines the complete joint
noncommutative distribution of the profile and the shift.

\begin{proposition}[Joint limit of the profile and the shift]
\label{prop:profile-shift-law}
Assume \eqref{H3} and $\sup_N\|D_N\|_{\mathrm{op}}<\infty$.  Then, for
every noncommutative $^*$-polynomial $P$,
\[
 \trN P(D_N,S_N,S_N^*)
 \longrightarrow
 \int_{\mathbb R}\int_{\mathbb T}
 P(t,z,\overline z)\,dm_{\mathbb T}(z)\,d\mu_T(t),
\]
where $m_{\mathbb T}$ denotes Haar probability measure on the unit circle.
Equivalently, $(D_N,S_N)$ converges in normalized joint $^*$-moments to
$(T,U)$, where $U$ is Haar unitary, $T$ commutes with $U$, and the two
variables are classically independent.
\end{proposition}

\begin{proof}
The matrices $S_N^*S_N$ and $S_NS_N^*$ differ from the identity by rank
one.  Cancelling adjacent occurrences of $S_N$ and $S_N^*$ in a fixed word
therefore changes its normalized trace by $O(N^{-1})$.  Moreover, for
$j\in\mathbb Z$ set
\[
 V_{N,j}:=
 \begin{cases}
  S_N^j,&j\ge0,\\
  (S_N^*)^{-j},&j<0.
 \end{cases}
\]
Then $\trN(V_{N,j})=0$ for every fixed nonzero $j$ and all sufficiently
large $N$.  Thus $S_N$ converges in $^*$-moments to a Haar unitary.  By
\eqref{eq:commutator-two} and the tracial Cauchy--Schwarz inequality, moving
one occurrence of $D_N$ across one occurrence of $S_N$ or $S_N^*$ changes
the normalized trace of a fixed word by $o(1)$.  Repeating this operation
and then making the preceding cancellations reduces the word, up to
$o(1)$, to $D_N^rV_{N,j}$.  Its normalized trace is zero when $j\ne0$,
whereas for $j=0$ it converges to
$\int t^r\,d\mu_T(t)$ by the marginal consequence of \eqref{H3} and the
uniform bound.  This is exactly the displayed product state.
\end{proof}

\medskip
We next answer a more general question.  Put
\[
 B_N:=\operatorname{diag}(b_1,\ldots,b_{N-1},0).
\]
After reversal of the basis, the zero-diagonal matrices have the exact
representations
\begin{equation}\label{eq:shift-representation}
 \widehat X_N=S_NB_N+B_NS_N^*,
 \qquad
 \widehat X_N^\sigma=S_ND_NB_N+B_ND_NS_N^*.
\end{equation}
Thus the deformation is an algebraic expression in the profile, the shift,
and an independent i.i.d. diagonal field.  It is not literally the product
$D_N\widehat X_N$, which would generally be non-Hermitian.

Whenever the profiles are uniformly bounded and $(D_N,S_N)$ has a limiting
joint $^*$-distribution, the same data encode a stationary law.  If
$M:=\sup_N\|D_N\|_{\mathrm{op}}$, define, for fixed $r$,
\[
 \rho_{N,r}:=\frac1{N-2r}\sum_{p=r+1}^{N-r}
 \delta_{(\sigma_{p-r,N},\ldots,\sigma_{p+r,N})}.
\]
Every polynomial moment of $\rho_{N,r}$ is a mixed $^*$-moment of
$(D_N,S_N)$ up to an $O_r(N^{-1})$ boundary term.  Joint $^*$-moment
convergence and Stone--Weierstrass therefore give a limiting measure
$\rho_r$ on $[-M,M]^{2r+1}$.  These limits are consistent and shift
invariant, and the Kolmogorov extension theorem therefore gives a
stationary law $\rho$ on $[-M,M]^{\mathbb Z}$.

\begin{theorem}[Shift-decorated convergence]
\label{thm:shift-decorated}
Let $(b_k)_{k\ge1}$ be i.i.d.\ real random variables, independent of the
deterministic profile, and suppose that
\[
 \sup_{N\ge1}\|D_N\|_{\mathrm{op}}<\infty,
 \qquad \|b_1\|_\infty<\infty,
\]
and that $(D_N,S_N)$ converges in normalized joint $^*$-moments.  Then the
empirical spectral distribution of
\[
 J_N:=S_ND_NB_N+B_ND_NS_N^*
\]
converges weakly almost surely to a deterministic compactly supported
probability measure $\mu_{\rho,b}$, and the expected ESD converges weakly
to the same measure.

More explicitly, for a nonzero finitely supported family
$\ell=(\ell_i)_{i\in\mathbb Z}$ of nonnegative integers, put
\begin{equation}\label{eq:profile-window-moment}
 L_\ell:=\lim_{N\to\infty}\frac1N
 \sum_{\substack{1\le p\le N\\
                  1\le p+i\le N,\ i\in\operatorname{supp}\ell}}
 \prod_{i\in\mathbb Z}\sigma_{p+i,N}^{\ell_i}.
\end{equation}
Set $L_0:=1$.  The limits in \eqref{eq:profile-window-moment} exist by the
joint $^*$-moment hypothesis.  For every $q\ge0$,
\begin{equation}\label{eq:general-shift-moments}
 m_q(\mu_{\rho,b})
 =\sum_{\gamma\in\mathcal C_q}
   L_{\ell(\gamma)}
   \prod_{i\in\mathbb Z}
   \E\bigl[b_1^{\ell_i(\gamma)}\bigr].
\end{equation}
If \eqref{H3} holds, then
$L_{\ell(\gamma)}=\E T^q$, and the conclusion reduces to
\[
 \mu_{\rho,b}=\mathcal L(TX),
\]
where $X\sim\mu_b$ and $T,X$ are independent.
\end{theorem}

\begin{proof}
We first explain why the deterministic limits in
\eqref{eq:profile-window-moment} are part of the joint law of $(D_N,S_N)$.
For $i\in\mathbb Z$, define the shifted diagonal matrix
\[
 D_N^{[i]}:=
 \begin{cases}
  (S_N^*)^iD_NS_N^i,&i\ge0,\\
  S_N^{-i}D_N(S_N^*)^{-i},&i<0.
 \end{cases}
\]
Away from $O_\ell(1)$ boundary indices, the $p$th diagonal entry of
$D_N^{[i]}$ is $\sigma_{p+i,N}$.  It follows that
\[
 \trN\prod_i(D_N^{[i]})^{\ell_i}
 =\frac1N
 \sum_{\substack{1\le p\le N\\
                  1\le p+i\le N,\ i\in\operatorname{supp}\ell}}
 \prod_i\sigma_{p+i,N}^{\ell_i}.
\]
The left-hand side is a mixed $^*$-moment of $(D_N,S_N)$, which proves the
existence of $L_\ell$.

The case $q=0$ is immediate, so assume $q\ge1$.  We now expand the trace
of $J_N^q$.  Exactly as in
\eqref{eq:path-trace-scale}, a contribution to a diagonal entry is indexed
by a closed nearest-neighbor path $\gamma\in\mathcal C_q$.  Taking
expectations and grouping repeated crossings of each edge gives
\begin{align*}
 \E\trN(J_N^q)
 &=\sum_{\gamma\in\mathcal C_q}
   \left[
    \frac1N\sum_{p\in I_N(\gamma)}
    \prod_i\sigma_{p+i,N}^{\ell_i(\gamma)}
   \right]
   \prod_i\E[b_1^{\ell_i(\gamma)}]\\
 &\longrightarrow
   \sum_{\gamma\in\mathcal C_q}
   L_{\ell(\gamma)}
   \prod_i\E[b_1^{\ell_i(\gamma)}].
\end{align*}
The admissible bases in $I_N(\gamma)$ differ from those in
\eqref{eq:profile-window-moment} by only $O_q(1)$ boundary indices, so the
bracketed term converges to $L_{\ell(\gamma)}$.
This is \eqref{eq:general-shift-moments}.

For almost-sure convergence, the contribution rooted at $p$ is a bounded
function of $b_{p-q},\ldots,b_{p+q}$.  After centering, these variables
have a dependency range depending only on $q$.  In the fourth-moment
expansion, a term vanishes whenever one of its four root neighborhoods is
disjoint from the other three; only $O_q(N^2)$ quadruples remain.  Hence
\[
 \E\left|\trN(J_N^q)-\E\trN(J_N^q)\right|^4=O_q(N^{-2}).
\]
Borel--Cantelli proves almost-sure convergence for every $q$.  Finally,
\[
 \|J_N\|_{\mathrm{op}}
 \le2\sup_N\|D_N\|_{\mathrm{op}}\,\|b_1\|_\infty,
\]
so the limiting moments determine a compactly supported probability
measure and moment convergence implies weak convergence.

Under \eqref{H3}, Lemma~\ref{lem:local-block} and boundedness give
\[
 \frac1N\sum_{p\in I_N(\gamma)}
 \prod_i\sigma_{p+i,N}^{\ell_i(\gamma)}
 \longrightarrow\E T^{\sum_i\ell_i(\gamma)}=\E T^q.
\]
Substitution in \eqref{eq:general-shift-moments} yields
$m_q(\mu_{\rho,b})=\E T^q\,m_q(\mu_b)$, which identifies the limit as
$\mathcal L(TX)$.
\end{proof}

The theorem has an equivalent stationary-operator interpretation.  Let
$(\xi_j)_{j\in\mathbb Z}$ have the stationary law $\rho$ constructed
above.  On an independent probability space let
$(\beta_j)_{j\in\mathbb Z}$ be i.i.d. with the law of $b_1$, put
$c_j:=\xi_j\beta_j$, and define the Jacobi operator
\[
 (\mathcal Jf)(j)=c_jf(j+1)+c_{j-1}f(j-1).
\]
Then, for every bounded continuous function $h$, the limiting measure is
the averaged spectral measure at the root,
\begin{equation}\label{eq:general-density-of-states}
 \int h(x)\,d\mu_{\rho,b}(x)
 =\E\langle e_0,h(\mathcal J)e_0\rangle.
\end{equation}
This is the density-of-states formulation familiar for stationary Jacobi
operators; see, for example, \cite{CS83}.  Under \eqref{H3}, the stationary
profile is the constant process $\xi_j=T$, which is precisely why the
general limit collapses to the scale mixture $TX$.

The ESD of $D_N$ alone cannot contain this positional information.  For
example, the block profile consisting of $N/2$ entries equal to $1$
followed by $N/2$ entries equal to $2$, and the alternating profile
$1,2,1,2,\ldots$, have the same limiting ESD
$\tfrac12(\delta_1+\delta_2)$.  The block profile satisfies \eqref{H3},
whereas the alternating profile does not.  When $b_k\equiv1$, their
limiting fourth spectral moments are, respectively,
\[
 6\,\frac{1^4+2^4}{2}=51,
 \qquad
 2\,\frac{1^4+2^4}{2}+4(1^2 2^2)=33.
\]
The general theorem therefore confirms that the joint law of $(D_N,S_N)$,
rather than the spectral law of $D_N$ alone, is the natural deterministic
datum.

For the compact moment proof, convergence of all the closed-path averages
in \eqref{eq:profile-window-moment} is the proof-minimal condition.  Joint
$^*$-moment convergence of $(D_N,S_N)$ is a transparent sufficient
condition, and \eqref{H3} is the locally constant special case used in the
main theorem.  Without uniform boundedness, weak convergence in
\eqref{H3} need not imply convergence of polynomial moments; the clipping
argument in Theorem~\ref{T3} is what permits the unbounded case.

Finally, if $b$ and the profile are bounded, \eqref{H3} gives the useful
approximation
\[
 \left\|\widehat X_N^\sigma
 -\frac{D_N\widehat X_N+\widehat X_ND_N}{2}\right\|_{2,N}^2
 \le\frac{\|b_1\|_\infty^2}{2N}
 \sum_{k=1}^{N-1}|\sigma_{k+1,N}-\sigma_{k,N}|^2
 \longrightarrow0.
\]
Theorem~\ref{thm:joint-sigma-X} then gives another interpretation of the
limit $TX$ in the bounded locally constant case.

\subsection{Comparison with Voiculescu asymptotic freeness}

For dense Wigner matrices, or under Haar conjugation, a deterministic
matrix sequence with a limiting noncommutative distribution becomes
asymptotically free from the independent random matrix under the standard
boundedness or moment hypotheses; see \cite{VDN92,AGZ10}.  Dense randomness
or Haar conjugation mixes the coordinate basis.  A tridiagonal matrix,
instead, retains the geometry of the one-step shift, so the order of the
diagonal entries remains visible.

The appropriate deterministic datum here is therefore not merely the ESD
of $D_N$, but its joint behavior with $S_N$, as in
\eqref{eq:shift-mixed}.  Theorem~\ref{thm:shift-decorated} proves that
convergence of this joint $^*$-distribution is sufficient for convergence
of the bounded zero-diagonal deformed model; Lemma~\ref{deform} restores a
diagonal satisfying (H2).  Under \eqref{H3}, the profile--shift limit
collapses to the commuting pair $(T,U)$ of
Proposition~\ref{prop:profile-shift-law}, and
Theorem~\ref{thm:joint-sigma-X} gives a commuting, classically independent
pair $(T,X)$ rather than a free pair.  For several tridiagonal matrices,
the common shift instead produces the colored-path moments of
Theorem~\ref{joint}.  These facts explain both the analogy with and the
distinction from Voiculescu's theory.  Operator-valued freeness and traffic
distributions provide related languages for retaining such geometric
information; see \cite{Shl96,Male20}.

For comparison, an independent uniform random permutation of one diagonal
array produces classical random pairing; under the usual boundedness or
tightness hypotheses, its limiting joint law is the product of the two
marginal limits.  With no conjugation, the existing joint empirical pairing
is preserved and need not be determined by the two marginal ESDs.  Haar
conjugation is the regime associated with free independence.

\section*{Remarks and open questions}

\begin{remark}[Finite-band matrices]
For fixed-band matrices, the same framework leads to stationary operators
generated by finitely many powers of the shift.  The corresponding moments
are described by paths with a finite set of allowed steps; see
\cite{basak2011limiting}.
Extending the analytic arguments of this paper to growing bandwidths is a
separate problem.
\end{remark}

\begin{question}[Beyond locally constant profiles]
Suppose the empirical laws of finite windows
$(\sigma_{k-r,N},\ldots,\sigma_{k+r,N})$ converge to a nonconstant
stationary process rather than to $(T,\ldots,T)$.  The bounded case is
described by Theorem~\ref{thm:shift-decorated}; under what moment and
uniform-integrability hypotheses does the same conclusion continue to hold
for unbounded profiles and off-diagonal variables?
\end{question}

\begin{question}[Common-shift independence]
Can the colored-path functional of Theorem~\ref{joint} be characterized
axiomatically as a notion of independence for stationary banded operators?
In particular, what is its precise relation to operator-valued or
traffic-type independence in the sense of \cite{Male20}?
\end{question}

\begin{question}[Eigenvector geometry]
How does the distribution of the eigenvector matrix of a tridiagonal
ensemble reflect the common-shift structure?  It would be interesting to
identify regimes interpolating between coordinate pairing and the
Haar-mixing behavior underlying Voiculescu asymptotic freeness.
\end{question}

\begin{question}[Gaussian field]
To better understand the Stieltjes transform of certain tridiagonal models,
one could aim to establish a central limit theorem on the entire upper
half-plane, yielding a Gaussian field under suitable assumptions.  Related
Gaussian and log-correlated field asymptotics for the characteristic
polynomial of the $G\beta E$ were obtained in \cite{LP25}.  Building on the
Gaussian CLT results in \cite[Theorems~6--7]{P09}, one might hope to obtain a
Stieltjes-transform field through a Gaussian analytic function framework.
\end{question}

\section{Acknowledgements}
We thank the reviewers for comments that substantially improved the paper,
and Thomas Buc-d'Alch\'e for his interest in and discussions of the topics
considered here.

\providecommand{\bysame}{\leavevmode\hbox to3em{\hrulefill}\thinspace}
\providecommand{\MR}{\relax\ifhmode\unskip\space\fi MR }
\providecommand{\MRhref}[2]{%
  \href{http://www.ams.org/mathscinet-getitem?mr=#1}{#2}
}
\providecommand{\href}[2]{#2}

\end{document}